\documentclass[11pt]{article}

\usepackage{authblk} % Adds affiliation

\usepackage[T1]{fontenc}
\usepackage[utf8]{inputenc}
\usepackage[english]{babel}
\usepackage{enumerate,amsmath,amsthm}
\usepackage{amssymb}
\usepackage{bookman}
\usepackage{float} % precise float package
\usepackage{enumitem,etoolbox,hyperref}
\usepackage{eulervm} % unbelievable font!
\usepackage{tabu}
\usepackage{multirow}
\usepackage{graphicx} %% Pictures
\usepackage{wrapfig} %% Pictures in text
\usepackage[font=small,labelfont=bf]{caption} % smaller font in figure labels
\usepackage{mathtools}% for bigtimes
\graphicspath{ {pictures/} } %% Path to pictures directory
\usepackage{placeins} % FloatBarrier
\allowdisplaybreaks % Allows breaking equation in align in multiple pages
\usepackage{tikz} % Graph and figure drawing
\tikzset{
    vertex/.style={minimum size=1.5em},
    edge/.style={->,> = latex'}
    main/.style = {draw,circle},
    every picture/.style=thick
}
\usetikzlibrary{arrows}% in tikz graphs
\usepackage{xcolor} %for colors in tables
\usepackage{colortbl} %colors in table cells
\usepackage{arydshln} % hdashed line
\definecolor{Gray}{gray}{0.9} %specific color in table cells

\newcommand{\ru}[1]{\rule{0pt}{#1 em}}%changing 
\input{mycommands.sty}

\input{tikzgraphs.sty}

\usepackage{comment}

\newlength\titlebox 

\theoremstyle{plain}
\newtheorem{theorem}{Theorem}%[section]

\newtheorem{corollary}[theorem]{Corollary}
\newtheorem{proposition}[theorem]{Proposition}

\newtheorem*{no-lemma}{Lemma}

\theoremstyle{definition}
\newtheorem{definition}[theorem]{Definition}

\theoremstyle{remark}
\newtheorem{remark}[theorem]{Remark}
\newtheorem{example}[theorem]{Example}

\input{amsbibstyle.sty}
\title{On the second moment of the determinant of random symmetric, Wigner, and Hermitian matrices}
\author[1]{Dominik Beck}
\affil[1]{\small Charles University, Faculty of Mathematics and Physics, Prague}
\affil[ ]{\textit {\href{mailto:beckd@karlin.mff.cuni.cz}{\tt beckd@karlin.mff.cuni.cz}}}
\author[2]{Zelin Lv}
\author[2]{Aaron Potechin}
\affil[2]{\small The University of Chicago}
\affil[ ]{\textit {\href{mailto:zlv@uchicago.edu}{\tt zlv@uchicago.edu}}}
\affil[ ]{\textit {\href{mailto:potechin@uchicago.edu}{\tt potechin@uchicago.edu}}}
\date{\today}
\newenvironment{acknowledgements}
    {\begin{center} \bfseries \scriptsize Acknowledgements \end{center}%
    \begin{quote}\scriptsize}
    {\end{quote}}

\begin{document}
\maketitle
\begin{abstract}
In this paper, we analyze the second moment of the determinant of random symmetric, Wigner, and Hermitian matrices. Using analytic combinatorics techniques, we determine the second moment of the determinant of Hermitian matrices whose entries on the diagonal are i.i.d and whose entries above the diagonal are i.i.d. and have real expected values. Our results extend previous work analyzing the second moment of the determinant of symmetric and Wigner matrices, providing a unified approach for this analysis.
\end{abstract}

\vspace{0.5em}

\begin{acknowledgements}
Dominik Beck was supported by the Charles University, project GA UK No. 71224 and by Charles University Research Centre program No. UNCE/24/SCI/022. He would also like to acknowledge the impact of the Dual Trimester Program: ``Synergies between modern probability, geometric analysis and stochastic geometry'' on him and this research. Zelin Lv and Aaron Potechin were supported by NSF grant CCF-2008920.
\end{acknowledgements}

\newpage
\tableofcontents

\section{Introduction}
Random matrix theory, which studies the properties of matrices with random entries, has found widespread applications across numerous fields, including physics, mathematics, engineering, and computer science. The determinant, a fundamental invariant of a matrix, plays a crucial role in understanding the properties of random matrices. One way to analyze the determinant of a random $n \times n$ matrix $M$ is to analyze $\Exx\left[\det(M)^k\right]$, i.e., the $k$th moment of the determinant of $M$. In this work, we delve into the second moments of the determinants of random symmetric matrices, Wigner matrices, and Hermitian matrices.

In the study of moments of random determinants, the main idea employed by previous research is the reduction of the problem of computing these moments to counting the number of permutation tables subject to certain constraints \cite{beck2022fourth,LP2022}. For random asymmetric matrices where the entries are i.i.d. (independent and identically distributed) with mean $0$ and variance $1$, Tur\'{a}n observed that the second moment of the determinant is $n!$, Nyquist, Rice, and Riordan \cite{NRR54} determined the fourth moment of the determinant, and the authors \cite{LP2022} determined the sixth moment of the determinant. For the special case when the entries of the random matrix are Gaussian, several works \cite{forsythe1952extent, NRR54} showed that the $k$th moment of the determinant is $\prod_{j=0}^{\frac{k}{2}-1}{\frac{(n + 2j)!}{(2j)!}}$ for even $k$ (when $k$ is odd the $k$th moment is $0$). For random symmetric matrices whose entries on or above the diagonal are i.i.d. with mean $0$, the second moment of the determinant was determined by Zhurbenko \cite{zhurbenko1968moments}. 

Several variations of this question have also been analyzed. For $p \times n$ random matrices $M$, while the determinant of $M$ is not defined, we can instead consider $\Exx\left[\det (M\!M^\top )^{\!\frac{k}{2}}\right]$. Dembo \cite{dembo1989random} determined $\Exx\left[\det (M\!M^\top )^{\!2}\right]$ for random $p \times n$ matrices $M$ with i.i.d. entries with mean $0$ and variance $1$. Recently, Beck \cite{beck2022fourth} generalized the results of Nyquist, Rice, and Riordan \cite{NRR54} and Dembo \cite{dembo1989random} to the setting where $M$ has i.i.d. entries from an arbitrary distribution (which may not have mean $0$).

In this work, we extend the techniques used in these papers to determine the second moment of the determinant of random Hermitian matrices where the diagonal entries are i.i.d., the entries above the diagonal are i.i.d., and the expected value of all entries is real. To describe our results, we will first define the symmetric, Wigner, and Hermitian matrices we analyze and describe what was known about the second moment of the determinant of these matrices. We will then state our main results.

\subsection{Symmetric, Wigner, and Hermitian matrices}
\paragraph{Symmetric matrices}
Let $X$ be a symmetric matrix such that the entries of $X$ on or above the diagonal are i.i.d. In other words, the entries $\{X_{ij}: i \leq j \in [n]\}$ are i.i.d. and $X_{ij} = X_{ji}$ for all $i,j \in [n]$. Here and for the remainder of this paper, $[n] = \{1,2,\ldots,n\}$. In matrix notation,
\begin{equation}
X =
    \begin{pmatrix} 
        X_{11} & X_{12} & X_{13} & \dots  & X_{1n}\\
        X_{12} & X_{22} & X_{23} & \dots  & X_{2n}\\
        X_{13} & X_{23} & X_{33} & \dots  & X_{3n}\\
        \vdots & \vdots & \vdots & \ddots & \vdots\\
        X_{1n} & X_{2n} & X_{3n} & \dots  & X_{nn}
    \end{pmatrix}.
\end{equation}
\begin{definition}
Let $m_q = \Exx\left[X_{ij}^q\right]$ be the $q$th moment of the entries of $X$.
\end{definition}
\begin{definition}
We define $f_k(n) = \Exx[\det(X)^k]$ to be the $k$th moment of the determinant of $X$ and we define $F_k(t) = \sum_{n=0}^{\infty}{f_k(n)\frac{t^n}{n!}}$ to be the corresponding exponential generating function (EGF).
\end{definition}

For the case when $k = 2$, Zhurbenko \cite{zhurbenko1968moments} found that when $m_1 = 0$,
\begin{equation}
        F_2(t) = \frac{\exp\left(\frac{(m_4-3m_2^2) t^2}{2}-m_2 t\right)}{\left(1-m_2 t\right){}^2 \sqrt{1-m_2^2 t^2}}.
    \end{equation}

\paragraph{Wigner matrices}
Wigner matrices are a subclass of random Hermitian matrices where the real and imaginary parts are independent. In particular, for a Wigner matrix $Z$, we have that for all $j,j' \in [n]$, $Z_{jj'} = Z^{\mathrm{Re}}_{jj'} + \mathbf{i} Z^{\mathrm{Im}}_{jj'}$ where $\mathbf{i} = \sqrt{-1}$ is the imaginary unit and
\begin{enumerate}
\item $\{Z^{\mathrm{Re}}_{jj'}: j < j' \in [n]\}$ are i.i.d., $\{Z^{\mathrm{Im}}_{jj'}: j < j' \in [n]\}$ are i.i.d., $\{Z^{\mathrm{Re}}_{jj}: j \in [n]\}$ are i.i.d., and $Z^{\mathrm{Im}}_{jj} = 0$ for all $j \in [n]$.
\item For all $j > j' \in [n]$, $Z^{\mathrm{Re}}_{jj'} = Z^{\mathrm{Re}}_{j'j}$ and $Z^{\mathrm{Im}}_{jj'} = -Z^{\mathrm{Im}}_{j'j}$.
\end{enumerate}
In matrix notation,
\begin{equation}
Z =
    \begin{pmatrix} 
        Z_{11} & Z^{\mathrm{Re}}_{12} + \mathbf{i} Z^{\mathrm{Im}}_{12} & Z^{\mathrm{Re}}_{13} + \mathbf{i} Z^{\mathrm{Im}}_{13} & \dots  & Z^{\mathrm{Re}}_{1n} + \mathbf{i} Z^{\mathrm{Im}}_{1n}\\
        Z^{\mathrm{Re}}_{12} - \mathbf{i} Z^{\mathrm{Im}}_{12}  & Z_{22} & Z^{\mathrm{Re}}_{23} + \mathbf{i} Z^{\mathrm{Im}}_{23}  & \dots  & Z^{\mathrm{Re}}_{2n} + \mathbf{i} Z^{\mathrm{Im}}_{2n} \\
        Z^{\mathrm{Re}}_{13} - \mathbf{i} Z^{\mathrm{Im}}_{13}  & Z^{\mathrm{Re}}_{23} - \mathbf{i} Z^{\mathrm{Im}}_{23}  & Z_{33} & \dots  & Z^{\mathrm{Re}}_{3n} + \mathbf{i} Z^{\mathrm{Im}}_{3n} \\
        \vdots & \vdots & \vdots & \ddots & \vdots\\
        Z^{\mathrm{Re}}_{1n} - \mathbf{i} Z^{\mathrm{Im}}_{1n}  & Z^{\mathrm{Re}}_{2n} - \mathbf{i} Z^{\mathrm{Im}}_{2n}  & Z^{\mathrm{Re}}_{3n} - \mathbf{i} Z^{\mathrm{Im}}_{3n}  & \dots  & Z_{nn}
    \end{pmatrix}.
\end{equation}

\begin{definition}
For $j < j' \in [n]$, let $\mu_q = \Exx[(Z^{\mathrm{Re}}_{jj'})^q]$ be the $q$th moment of the real part of $Z_{jj'}$ and let $\nu_q = \Exx[(Z^{\mathrm{Im}}_{jj'})^q]$ be the $q$th moment of the imaginary part of $Z_{jj'}$. For $j \in [n]$, let $\kappa_q = \Exx[(Z_{jj})^q]$ be the $q$-th moment of $Z_{jj}$.

In this paper, we assume that $\mu_1 = \nu_1 = \kappa_1 = 0$ as we will handle the expected value of the entries of Wigner matrices separately.
\end{definition}
Instead of analyzing $\Exx[(\det Z)^2]$, we analyze a more general expression called the \emph{characteristic correlation function} where we can adjust the expected values of the entries of each copy of $Z$.
\begin{definition}
We define the $n \times n$ matrix $J$ to be the all ones matrix. In other words, $J_{jj'} = 1$ for all $j,j' \in [n]$.
\end{definition}
\begin{definition}
We define the \emph{characteristic correlation function} to be
\begin{equation}
h_k(n,a_1,\ldots,a_k,c_1,\ldots,c_k) = \Exx \prod_{r=1}^k\det(Z + {a_r}J + {c_r}Id)
\end{equation}
and we define its EGF (exponential generating function) to be
\begin{equation}
H_k(t,a_1,\ldots,a_k,c_1,\ldots,c_k) = \sum_{n=0}^\infty \frac{t^n}{n!} h_k(n,a_1,\ldots,a_k,c_1,\ldots,c_k).
\end{equation}
\end{definition}
\begin{example}
When $n = 2$, we have that 
\[
f_2(n) = \Exx \left[\det\left(
\begin{matrix} 
X_{11} & X_{12}\\
X_{21} & X_{22}\\
\end{matrix}
\right)^{2}
\right]
\]
while $h_2(n,a_1,a_2,c_1,c_2)$ is equal to 
\[
\Exx \left[\det\left(\left(
\begin{matrix} 
Z^{\mathrm{Re}}_{11}& Z^{\mathrm{Re}}_{12} + Z^{\mathrm{Im}}_{12}i\\
Z^{\mathrm{Re}}_{12} - Z^{\mathrm{Im}}_{12}i & Z^{\mathrm{Re}}_{22}
\end{matrix}\right)  + {a_1}J + {c_1}Id
\right)\det\left(\left(
\begin{matrix} 
Z^{\mathrm{Re}}_{11}& Z^{\mathrm{Re}}_{12} + Z^{\mathrm{Im}}_{12}i\\
Z^{\mathrm{Re}}_{12} - Z^{\mathrm{Im}}_{12}i & Z^{\mathrm{Re}}_{22}
\end{matrix}\right)  + {a_2}J + {c_2}Id
\right)
\right]
\]
\end{example}

In general, the characteristic correlation function $h_k(n,a_1,\ldots,a_k)$ is a polynomial in $\{\mu_q, \nu_q: q \in [2,2k]\} \cup \{\kappa_q: q \in [2,k]\} \cup \{a_r,c_r: r \in [k]\}$. Note that the symmetric matrices described above are a special case of Wigner matrices where $\kappa_q =\mu_q$, $\nu_q = 0$, $a_1 = \cdots = a_k = a$, and $c_1 = \cdots = c_k = 0$. In this case, we can set $m_1 = a$ and $m_q = \Exx\left[(Z_{ij} + m_1)^q\right]$ so $h_2(n,a,a,0,0)$ coincides with $f_2(n)$ where $m_1 = a$, $m_2 = m_1^2 + \mu_2$, $m_3 = m_1^3 + 3m_1\mu_2 + \mu_3$, and $m_4 = m_1^4 + 6m_1^2\mu_2 + 4m_1\mu_3 + \mu_4$.

For $k = 2$, a special case of $H_2(t,a_1,a_2,c_1,c_2)$ was derived by Götze and Kösters \cite{gotze2009second}. They found that when $\mu_2 = \nu_2 = \frac{\kappa_2}{2}$, $\mu_3 = \nu_3 = 0$, $\mu_4 = \nu_4 = \frac{\kappa_4}{4}$,
\begin{equation}
        H_2(t,0,0,c_1,c_2) = \frac{\exp \left(c_1 c_2 t -\frac{t^2 \left(c_1^2+c_2^2\right) \kappa _2}{2 \left(1-t^2 \kappa _2^2\right)}+\frac{t^3 c_1 c_2 \kappa_2^2}{1-t^2 \kappa _2^2}+ \frac{t^2}{4} \left(\kappa_4 - 3 \kappa_2^2\right)\right)}{\left(1- \kappa_2 t\right) \sqrt{1-t^2 \kappa_2^2}}.
    \end{equation}

Shortly afterwards, K\"{o}sters \cite{kosters2008second} found $H_2(t,0,0,c_1,c_2)$ for the special case where $\mu_2 = \frac{\kappa_2}{2}$, $\mu_3 = 0$, $\nu_q = 0$ for all $q \in [4]$, and $\mu_4 = \frac{\kappa_4}{4}$. Note that this corresponds to the case when $Z = X + X^T$ for an asymmetric random matrix $X$ with i.i.d. real entries whose odd moments are $0$.

Not much is known about higher moments in general, but there is an important result if the random Wigner matrix is drawn from the Gaussian Unitary Ensemble (GUE). In this case, each diagonal entry $Z_{jj}$ is a standard normal variable while the entries $\{Z_{jj'}: j < j'\}$ above the diagonal are standard complex normal random variables, which means that $Z_{jj'} \sim \mathcal{N}(0,\frac{1}{2})+\mathbf{i} \mathcal{N}(0,\frac{1}{2})$. For the GUE, Mehta and Normand \cite{mehta1998probability} found the value of $h_k(n,0,\ldots,0,0,\ldots, 0)$ for any $k$ and $n$.
\paragraph{Hermitian matrices}
Finally, we consider the more general case of random Hermitian matrices where the expected value of each entry is real. In particular, we consider $n \times n$ matrices $Z$ where
\begin{enumerate}
\item The entries $\{Z_{jj'}: j < j' \in [n]\}$ are i.i.d. and for all $j < j' \in [n]$, $Z_{j'j} = \overline{Z_{jj'}}$. Note that these entries are complex numbers.
\item The entries $\{Z_{jj}: j \in [n]\}$ are real and i.i.d..
\end{enumerate}
In matrix notation,
\begin{equation}
Z =
    \begin{pmatrix} 
        Z_{11} & Z_{12} & Z_{13} & \dots  & Z_{1n}\\
        \overline{Z}_{12} & Z_{22} & Z_{23} & \dots  & Z_{2n}\\
        \overline{Z}_{13} & \overline{Z}_{23} & Z_{33} & \dots  & Z_{3n}\\
        \vdots & \vdots & \vdots & \ddots & \vdots\\
        \overline{Z}_{1n} & \overline{Z}_{2n} & \overline{Z}_{3n} & \dots  & Z_{nn}
    \end{pmatrix}.
\end{equation}
\begin{definition}
For $j < j'$, we let $\lambda_{pq} = \Exx\left[Z_{jj'}^p\overline{Z_{jj'}^q}\right]$ be the moments of the entries of $Z$ which are above the diagonal. Note that $\lambda_{pq}=\overline{\lambda}_{qp}$. As before, for $j \in [n]$ we let $\kappa_q = \Exx [Z_{jj}^q]$. Unless noted otherwise, we assume that $\lambda_{10} = \kappa_1 = 0$.
\end{definition}
As before, we take the \emph{characteristic correlation function} to be
$h_k(n,a_1,\ldots,a_k,c_1,\ldots,c_k) = \Exx \prod_{r=1}^k\det(Z + {a_r}J + {c_r}Id)$
and we take its EGF (exponential generating function) to be \\
$H_k(t,a_1,\ldots,a_k,c_1,\ldots,c_k) = \sum_{n=0}^\infty \frac{t^n}{n!} h_k(n,a_1,\ldots,a_k,c_1,\ldots,c_k)$.
\begin{remark}
Wigner matrices are a special case of random Hermitian matrices where we have that
\begin{equation}
\lambda_{pq} = \sum_{k=0}^{p+q} \mu_{k} \nu_{p+q-k} \mathbf{i}^{p+q-k} \sum_{j=0}^q \binom{p}{k-j} \binom{q}{j} (-1)^{q-j}.
\end{equation}
\end{remark}
The following Table lists the moments $\lambda_{pq}$ for Wigner matrices for $p,q \leq 3$.
\begin{table}[ht]
\centering
\begin{tabular}{|c|c|c|c|c|c|}
\hline
 \multicolumn{2}{|c|}{\multirow{2}{*}{$\lambda_{pq}$}} & \multicolumn{4}{c|}{$q$}  \\ \cline{3-6}
   \multicolumn{2}{|c|}{}  & $0$ & $1$ & $2$ & $3$ \\
\hline
\multirow{4}{*}{$p$} & $0$ & $1$ & $0$ & $\mu_2-\nu _2$ & $\mu _3+\mathbf{i} \nu _3$ \\
 & $1$ & $0$ & $\mu_2+\nu_2$ & $\mu_3-\mathbf{i} \nu_3$ & $\mu_4-\nu_4$ \\
 & $2$ & $\mu _2-\nu _2$ & $\mu _3+\mathbf{i} \nu _3$ & $\mu_4+2\mu_2 \nu_2+\nu_4$ & $\mu_5+2\mu_3 \nu_2- 2\mathbf{i}\mu_2 \nu_3-\mathbf{i} \nu_5$ \\
 & $3$ & $\mu _3-\mathbf{i} \nu _3$ & $\mu _4-\nu _4$ & $\mu_5+2 \mu_3 \nu _2+2 \mathbf{i} \mu_2 \nu_3+\mathbf{i} \nu _5$ & $\mu_6+3\mu_4\nu_2+3\mu_2 \nu_4+\nu_6$ \\
\hline
\end{tabular}
\caption{Values of $\lambda_{pq}$ in terms of $\mu_p$ and $\nu_q$ for Wigner matrices}
\label{tab:Wigner}
\end{table}
\subsection{Main results}
In this paper, we generalize all of the results mentioned above on the second moment of the determinant of random symmetric matrices and Wigner matrices. For symmetric matrices, we generalize Zhurbenko's result on the second moment of the determinant to handle all distributions of the entries $X_{ij}$.
\begin{theorem}
For random symmetric matrices, for any distribution of the entries $X_{ij}$,
\begin{equation*}
F_2(t) = \left(1+2 m_1 \mu _3 t^2+m_1^2 \left(\mu_3^2t^4+\frac{t}{1-\mu _2^2 t^2}+\frac{2 \mu _2 t^2}{1-\mu _2 t}\right)\right) \frac{\exp\left(\frac{(\mu _4-3\mu_2^2) t^2}{2}-\mu _2 t\right)}{\left(1-\mu _2 t\right){}^2 \sqrt{1-\mu _2^2 t^2}}.
\end{equation*}
\end{theorem}
For Wigner matrices, we obtain the following result.
\begin{theorem}[Wigner matrices]
For Wigner matrices, for any distributions of $\{Z^{\mathrm{Re}}_{jj'}: j < j' \in [n]\}$, $\{Z^{\mathrm{Im}}_{jj'}: j < j' \in [n]\}$, and $\{Z_{jj}: j \in [n]\}$ with expected value $0$,
\begin{align*}
& H_2(t,a_1,a_2,c_1,c_2) = \Bigg{[}1+(a_1 + a_2) \mu_3 t^2+\frac{(a_1c_2+a_2c_1) t}{1-(\mu_2+\nu_2)^2 t^2} -\frac{(\mu_2+\nu_2)(a_1c_1+a_2c_2) t^2}{1-(\mu_2+\nu_2)^2t^2}\\
&  + \left.  a_1 a_2 \Bigg{(}\frac{2t^2 \left(\mu_2-\mu_2^2 t+\nu_2^2 t\right)}{1-2\mu_2t+(\mu_2^2-\nu_2^2)t^2}+\frac{t}{1-(\mu_2+\nu_2)^2 t^2}+ \bigg{(}\mu_3 t^2 + \frac{c_2 t}{1\!-\!\left(\mu_2\!+\!\nu_2\right){}^2 t^2} \right.\\
&  - \frac{c_1 \left(\mu _2+\nu _2\right) t^2}{1\!-\!\left(\mu _2\!+\!\nu_2\right){}^2 t^2}\bigg{)} \bigg{(}\mu_3 t^2\!+\!\frac{c_1 t}{1\!-\!\left(\mu_2\!+\!\nu_2\right){}^2 t^2}\!-\!\frac{c_2 \left(\mu _2+\nu _2\right) t^2}{1\!-\!\left(\mu _2\!+\!\nu _2\right){}^2 t^2}\bigg{)}\Bigg{)}\Bigg{]}\\
& \frac{\exp \left(t \left(c_1 c_2\!+\!\kappa_2\!-\!2\mu_2\right)+\frac{1}{2} t^2 \left(\mu_4\!-\!3 \nu_2^2\!-\!3 \mu_2^2\!+\!\nu_4\right)-\frac{t^2}{2}\frac{(c_1^2+c_2^2) \left(\mu_2+\nu_2\right)}{1-t^2 \left(\mu _2+\nu _2\right){}^2}+\frac{c_1 c_2
   \left(\mu _2+\nu _2\right){}^2 t^3}{1-\left(\mu _2+\nu _2\right){}^2 t^2}\right)}{\left(1-2 \mu _2 t+\left(\mu _2^2-\nu _2^2\right)
   t^2\right) \sqrt{1-t^2 \left(\mu _2+\nu _2\right){}^2}}.
\end{align*}
\end{theorem}
\begin{corollary}
For the special case when $c_1 = c_2 = c$,
\begin{align*}
& \textstyle{H_2(t,a_1,a_2,c,c) = \left[1+\left(a_1+a_2\right) \left(\mu_3 t^2 + \frac{ct}{1 + (\mu_2+\nu_2)t}\right)+a_1 a_2 \left(\mu_3^2 t^4+\frac{2c \mu_3 t^3}{1+(\mu_2+\nu_2)t} +\frac{c^2 t^2}{\left(1+(\mu_2+\nu_2)t\right){}^2} \right. \right.}\\
& \textstyle{ \left. \left. + \frac{2 t^2 \left(\mu_2 - \mu_2^2 t +\nu_2^2 t\right)}{1\!-\!2 \mu_2t\!+\!\left(\mu_2^2\!-\!\nu_2^2\right) t^2} \!+\!\frac{t}{1\!-\!\left(\mu_2\!+\!\nu_2\right){}^2 t^2}\right)\right]}
\displaystyle \frac{\exp \left(t \left(c^2\!+\!\kappa_2\!-\!2\mu_2\right) \!+\! \frac{1}{2} t^2 \left(\mu_4\!-\!3 \nu_2^2 \!-\! 3\mu_2^2 \!+\! \nu_4\right) \!-\! \frac{t^2 c^2 \left(\mu_2 + \nu_2\right)}{1 + t \left(\mu_2 + \nu_2\right)}\right)}{\left(1 - 2\mu_2t + (\mu_2^2-\nu_2^2)t^2\right) \sqrt{1-\left(\mu_2+\nu_2\right){}^2 t^2}}.
\end{align*}
\end{corollary}
Finally, we generalize this result to random Hermitian matrices where the expected value of each entry is real.
\begin{theorem}[Hermitian matrices]
For random Hermitian matrices, for any distributions of the entries $\{Z_{jj}: j \in [n]\}$ and $\{Z_{jj'}: j < j' \in [n]\}$ with expected value $0$,
\begin{align*}
& H_2(t,a_1,a_2) = \Bigg{[}1+(a_1 + a_2) (\lambda_{21}+\lambda_{12}) \frac{t^2}{2} +\frac{(a_1c_2+a_2c_1) t}{1-\lambda_{11}^2 t^2} - \frac{\lambda_{11}(a_1c_1+a_2c_2) t^2}{1-\lambda_{11}^2t^2} \\
& + a_1 a_2 \Bigg{(}\frac{\lambda _{11} t^2}{1-\lambda _{11} t}+\frac{t}{1-\lambda _{11}^2 t^2}-t+\frac{e^{\left(\lambda_{02}-\lambda_{20}\right) t}-1}{\lambda_{02}-\lambda_{20} e^{\left(\lambda_{02}-\lambda_{20}\right) t}}+ \bigg{(}\left(\lambda_{21}+\lambda_{12}\right) \frac{t^2}{2}\\
& + \frac{c_2 t}{1\!-\!\lambda_{11}^2 t^2} - \,\frac{c_1 \lambda_{11} t^2}{1\!-\!\lambda_{11}^2 t^2} \bigg{)} \bigg{(}\left(\lambda_{21}+\lambda_{12}\right) \frac{t^2}{2}\!+\!\frac{c_1 t}{1\!-\!\lambda_{11}^2 t^2}\!-\!\frac{c_2 \lambda_{11} t^2}{1\!-\!\lambda_{11}^2 t^2}\bigg{)}\Bigg{)}\Bigg{]}\\
& \frac{\left(\lambda_{02}\!-\!\lambda _{20}\right) \exp \!\left(\!
   \left(c_1 c_2\!+\!\kappa_2\!-\!\lambda_{11}\!-\!\lambda_{20}\right)t\!+\!\frac{t^2}{2}\left(\lambda_{22}\!-\!2\lambda_{11}^2 \!-\! \lambda_{02}
   \lambda_{20} \right)\!+\!\frac{c_1 c_2 \lambda_{11}^2 t^3}{1\!-\!\lambda_{11}^2 t^2}\!-\!\frac{t^2}{2}\frac{\left(c_1^2\!+\!c_2^2\right) \lambda_{11}}{1\!-\!\lambda_{11}^2 t^2}\!\right)}{\left(1\!-\!
   \lambda_{11}t\right) \sqrt{1\!-\!\lambda_{11}^2t^2}
   \left(\lambda_{02}\!-\!\lambda_{20} e^{ \left(\lambda_{02}\!-\!\lambda_{20}\right)t}\right)}.
\end{align*}
\end{theorem}
\begin{remark}
The reason that we can adjust the expected value of the entries of $Z$ by $a$ when $a \in \mathbb{R}$ is that $det(M + aJ)$ is linear in $a$ (see Section \ref{sec:markedmultigraphs}). To handle random Hermitian matrices where the expected value of the off-diagonal entries is complex, we would need to analyze $det(M + b\mathbf{i}O)$ where $O$ is the matrix such that $O_{jj'} = -O_{j'j} = 1$ for all $j < j' \in [n]$ and $O_{jj} = 0$ for all $j \in [n]$. Since $det(M + b\mathbf{i}O)$ is nonlinear in $b$, our current techniques do not extend to this setting. 
\end{remark}
\subsection{Organization}
In section \ref{sec:prelim}, we give background on the analytic combinatorics tools we use in this paper. Interested readers can find more detailed information about analytic combinatorics in the book \cite{flajolet2009analytic} by Flajolet and Sedgewick. In Sections \ref{sec:symmetric}, \ref{sec:Wigner}, and \ref{sec:Hermitian}, we prove our results for symmetric matrices, Wigner matrices, and Hermitian matrices respectively.
\section{Preliminaries}\label{sec:prelim}
\subsection{Analytic combinatorics}
We follow the notation from the textbook \emph{Analytic combinatorics} \cite{flajolet2009analytic} by Flajolet and Sedgewick.

Let $\mathcal{A}$ be a set of objects with a given structure where each $\alpha \in \mathcal{A}$ has a size $|\alpha| \in \mathbb{N} \cup \{0\}$ and a weight $w(\alpha) \in \mathbb{C}$ (note that $w(\alpha)$ may be negative or even complex). We call $\mathcal{A}$ a combinatorial strucuture and view it in terms of the structure that its elements satisfy.

We say that $\mathcal{A}$ is labeled if each $\alpha \in \mathcal{A}$ is composed of atoms labeled by $[|\alpha|]= \{1,2,3,4,\ldots, |\alpha|\}$. Moreover, we assume that $\mathcal{A}_n = \{\alpha \in \mathcal{A}: |\alpha| = n \}$ is finite for all $n \geq 0$. We define $a_n = \sum_{\alpha \in \mathcal{A}: |\alpha| = n}{w(\alpha)}$ to be the total weight of the objects in $\mathcal{A}$ with size $n$.

Combinatorial structures can be composed together. One common composition is the \emph{star product}. Note that a tuple $(\alpha, \beta) \in \mathcal{A} \times \mathcal{B}$ cannot represent a labeled object of any structure as the atoms of $\alpha$ and $\beta$ are labeled by $[|\alpha|]$ and $[|\beta|]$, respectively. Relabeling our $\alpha$, $\beta$ as $\alpha'$, $ \beta'$ so that every number from $1$ to $|\alpha|+|\beta|$ appears once, we get a correctly labeled object. There are of course many ways how to relabel the objects. The canonical way is to use the \emph{star product}. We say $(\alpha' ,\beta') \in \alpha \star \beta$ if the new labels in both $\alpha'$ and $\beta'$ increase in the same order as in $\alpha$ and $\beta$ separately. An example is illustrated below.
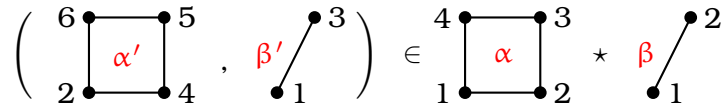
\begin{figure}[htb!]
    \centering
    \begin{tikzpicture}
        \node at (-0.9, 0.5) {$\Bigg{(}$};

        % Alpha'
        \draw (0, 0) rectangle (1, 1);
        \draw[fill=black] (0,0) circle (2pt);
        \draw[fill=black] (0,1) circle (2pt);
        \draw[fill=black] (1,1) circle (2pt);
        \draw[fill=black] (1,0) circle (2pt);    
        \node at (0.5, 0.5) {\textcolor{red}{$\alpha'$}};
        \node at (-0.3, 1) {6};
        \node at (1.3, 1) {5};
        \node at (-0.3, 0) {2};
        \node at (1.3, 0) {4};

        % ,
        \node at (1.8, 0.3) {$,$};

        % Beta'
        \draw (2.5, 0) -- (3, 1);
        \draw[fill=black] (2.5,0) circle (2pt);
        \draw[fill=black] (3,1) circle (2pt);
        \node at (2.4, 0.5) {\textcolor{red}{$\beta'
$}};
        \node at (3.3, 1) {3};
        \node at (2.8, 0) {1};

        \node at (3.7, 0.5) {$\Bigg{)}$};

        % Symbol
        \node at (4.3, 0.5) {$\in$};

        % Product
        \begin{scope}[shift={(5, 0)}]
            \draw (0, 0) rectangle (1, 1);
            \draw[fill=black] (0,0) circle (2pt);
            \draw[fill=black] (0,1) circle (2pt);
            \draw[fill=black] (1,1) circle (2pt);
            \draw[fill=black] (1,0) circle (2pt);    
            \node at (0.5, 0.5) {\textcolor{red}{$\alpha$}};
            \node at (-0.3, 1) {4};
            \node at (1.3, 1) {3};
            \node at (-0.3, 0) {1};
            \node at (1.3, 0) {2};

            % ,
            \node at (1.8, 0.5) {$\star$};

            \draw (2.5, 0) -- (3, 1);
            \draw[fill=black] (2.5,0) circle (2pt);
            \draw[fill=black] (3,1) circle (2pt);
            \node at (2.4, 0.5) {\textcolor{red}{$\beta$}};
            \node at (3.3, 1) {2};
            \node at (2.8, 0) {1};
        \end{scope}
    \end{tikzpicture}
    \caption{Star product}
    \label{fig:starprod}
\end{figure}

A key concept for labeled combinatorial structures is their exponential generating function (EGF for short) defined as
\begin{equation}
    A(t) = \sum_{n = 0}^{\infty}{\sum_{\alpha \in \mathcal{A}_n}{\frac{w(\alpha)t^n}{n!}}} = \sum_{n=0}^\infty{{a_n}\frac{t^n}{n!}}.
\end{equation}
These exponential generating functions encode the relationships between combinatorial structures (i.e., how they are composed). We can write the following relations.

\bgroup
\renewcommand{\arraystretch}{1.5}
\begin{table}[tbh!]
    \centering
    \begin{tabular}{|c|c|c|c|}
    \hline
        $\mathcal{C}$ & representation of $\mathcal{C}$ & $w_\mathcal{C}(\gamma), \gamma \in \mathcal{C}$ & $C(t)$ \\
       \hline
       \ru{2} $\mathcal{A} + \mathcal{B}$ & $\mathcal{A}\cup \mathcal{B}$ & $\begin{cases}
    w_\mathcal{A}(\gamma) & \text{if } \gamma \in \mathcal{A}, \\
    w_\mathcal{B}(\gamma) & \text{if } \gamma \in \mathcal{B}
\end{cases}$ & $A(t) + B(t)$\\
       $\mathcal{A} \star \mathcal{B}$ &  $\{\gamma \mid \gamma \in \alpha \star \beta, \alpha \in \mathcal{A}, \beta \in \mathcal{B}\}$ & $w_\mathcal{A}(\alpha)w_\mathcal{B}(\beta)$ & $A(t) B(t)$\\
       $\textsc{Seq}_k(\mathcal{A})$ & $\mathcal{A}^k \overset{\mathrm{def.}}{=} \underbrace{\mathcal{A} \star \mathcal{A} \star \cdots \star \mathcal{A}}_{k}$ & & $A^k(t)$\\
       $\textsc{Set}_k(\mathcal{A})$ & $\frac{1}{k!} \mathcal{A}^k$ & & $\frac{1}{k!} A^k(t)$\\
       $\textsc{Cyc}_k(\mathcal{A})$ & $\frac{1}{k} \mathcal{A}^k$ & & $\frac{1}{k} A^k(t)$\\
       $\textsc{Seq}(\mathcal{A})$ & $\sum_{k=0}^\infty\textsc{Seq}_k(\mathcal{A})$ & & $\frac{1}{1-A(t)}$\\
       $\textsc{Set}(\mathcal{A})$ & $\sum_{k=0}^\infty\textsc{Set}_k(\mathcal{A})$ & & $\exp{(A(t))}$\\
       $\textsc{Cyc}(\mathcal{A})$ & $\sum_{k=1}^\infty\textsc{Cyc}_k(\mathcal{A})$ & & $\ln \left(\frac{1}{1-A(t)}\right)$\\
    \hline
    \end{tabular}
    \caption{Composition of combinatorial structures and the corresponding exponential generating functions}
    \label{tab:EGFancom}
\end{table}
\egroup

The meaning of $\textsc{Seq}_k(\mathcal{A})$, $\textsc{Set}_k(\mathcal{A})$, and $\textsc{Seq}_k(\mathcal{A})$ is as follows.
\begin{itemize}
    \item $\textsc{Seq}_k(\mathcal{A})$ is shorthand for a \emph{sequence} and indeed it can be represented as (relabeled) $k$-tuples of objects taken from $\mathcal{A}$. Note that since everything is relabeled, even though $\alpha_i,\alpha_j$ might be the same for different $i, j$, the corresponding $\alpha_i'$, $\alpha_j'$ are always distinct. Formally, $\textsc{Seq}_k(\mathcal{A}) = \{(\alpha_1',\ldots,\alpha_k') \mid \alpha_i \in \mathcal{A}, i \in [k]\}$, where $(\alpha'_1,\ldots,\alpha'_k) \in \alpha_1 \star \cdots \star \alpha_k$.
    \item $\textsc{Set}_k(\mathcal{A})$ is a structure of \emph{sets} of $k$ relabeled elements, that is, the order of objects $\alpha_i'$ is irrelevant. Formally, $\textsc{Set}_k(\mathcal{A}) = \{\{\alpha_1',\ldots,\alpha_k'\} \mid \alpha_i \in \mathcal{A}, i \in [k]\}$. Alternatively, $\textsc{Set}_k(\mathcal{A})$ can be represented as the structure of classes of $k$-tuples in $\textsc{Seq}_k(\mathcal{A})$ which differ up to some permutation.
    \item $\textsc{Cyc}_k(\mathcal{A})$ represents the structure of classes of $k$-tuples in $\textsc{Seq}_k(\mathcal{A})$ which differ up to some cyclical permutation.
\end{itemize}

For completeness, we briefly explain these results. To see that the exponential generating function for $\mathcal{A} \star \mathcal{B}$ is $A(t) B(t)$, let $a_n$, $b_n$, and $c_n$ be the total weight of the objects of size $n$ in $\mathcal{A}$, $\mathcal{B}$, and $\mathcal{A} \star \mathcal{B}$ respectively. We have that
\[
c_n = \sum_{j=0}^{n}{\binom{n}{j}{a_j}b_{n-j}}
\]
so 
\[
C(t) = \sum_{n=0}^{\infty}{\frac{{c_n}t^n}{n!}} = \sum_{n=0}^{\infty}{\sum_{j=0}^{n}{\frac{{a_j}t^{j}}{j!} \cdot \frac{b_{n-j}t^{n-j}}{(n-j)!}}} = A(t) B(t)
\]
The generating functions for $\textsc{Seq}(\mathcal{A})$, $\textsc{Set}(\mathcal{A})$, and $\textsc{Cyc}(\mathcal{A})$ come from the Taylor series $\frac{1}{1-x} = \sum_{k=0}^{\infty}{x^{k}}$, $e^{x} = \sum_{k=0}^{\infty}{\frac{x^{k}}{k!}}$, and $-\ln(1-x) = \sum_{k=1}^{\infty}{\frac{x^k}{k}}$ respectively.

\begin{example}
Let $\mathcal{D}$ be the combinatorial structure of all derangements (i.e., permutations of $[n]$ where no element is mapped to itself). Any derangement can be decomposed into cycles of length at least two. Attaching a tag $u$ to each cycle so that the weight of each derangement is equal to $u^{\# \text{ of cycles}}$, we get a structure $\mathcal{D}_u$ which can be also constructed as follows:
\begin{equation}
\mathcal{D}_u = \textsc{Set}\left(u\,\textsc{Cyc}_{\geq 2}\left(\,
    \begin{tikzpicture}[baseline=-0.7ex, thick, main/.style = {draw,circle, inner sep = 2pt}, scale = 0.7]
    \node[main] (1) at (0,0) {$1$};
    \end{tikzpicture}
        \,\right)\right).
\end{equation}
In terms of generating functions,
\begin{equation}
    D_u(t) = \exp(-u t - u\ln(1-t)) = \frac{e^{-ut}}{(1-t)^u}.
\end{equation}
\end{example}

\section{Symmetric matrices}\label{sec:symmetric}

In this section, we analyze the second moment of the determinant of symmetric matrices. We start by showing how the problem can be reduced to counting weighted permutation tables and how these permutation tables can be represented using multi-graphs. We then illustrate these techniques by computing the first moment of the determinant of a random symmetric matrix. 

To analyze the second moment, we observe that the determinant of $X = Y + {m_1}J$ is linear in $m_1$. Representing each factor of $m_1$ by a marked edge, this means that it is sufficient to consider marked multi-graphs which have at most two marked edges. All such marked multi-graphs can be split into components of a few different types which allows us to determine the second moment of the determinant of a random symmetric matrix for arbitrary $m_1$.
\subsection{Permutation tables}
By definition, we have
\begin{equation}
f_k(n) = \Exx \left[\prod_{j=1}^k \sum_{\pi_j \in S_n} \sgn (\pi_j) \prod_{i=1}^n X_{i \pi_j(i)}\right] = \sum_{t \in F_{k,n}} w(t) \sgn (t),
\end{equation}
where 
\begin{enumerate}
\item $F_{k,n}$ is the set of $k$ by $n$ tables whose rows are permutations $\pi_j \in S_n$ (where $S_n$ is the set of all permutations of $[n]$). 
\item $\sgn (t) = \prod_{j=1}^{k}{\sgn(\pi_j)}$. In other words, the sign of $t$ is the product of the signs of the permutations corresponding to the rows of $t$. 
\item The weight $w(t)$ of $t$ is $w(t) = \Exx \left[\prod_{j=1}^{k}{{\prod_{i=1}^n X_{i \pi_j(i)}}}\right]$. In other words, $w(t)$ is the expected value of the product of the entries of $X$ corresponding to the entries of $t$.
\end{enumerate}
Note that, unlike in the asymmetric case, the weight $w(t)$ cannot be expressed simply as the product of the weights of the individual columns of $t$. Consider the example in Figure \ref{fig:tabsymminim} -- to keep track of the first indices of the variables $X_{ij}$, we have added the identity permutation as the \emph{zeroth} row as a reference.
\begin{figure}[H]
\centering
    \begin{tabular}{|c|c|c|c|c|c|c|c|c|}
    \hline
        1 & 2 & 3 & 4 & 5 & 6 & 7 & 8 & 9 \\
    \hdashline
        4 & 9 & 1 & 7 & 5 & 6 & 8 & 3 & 2 \\
        1 & 9 & 8 & 7 & 6 & 3 & 4 & 5 & 2 \\
    \hline
    \end{tabular}
\caption{An example of a table $t \in  F_{2,9}$ with weight $w(t) = m_1^9 m_2 m_3 m_4$}
\label{fig:tabsymminim}
\end{figure}
\noindent
Ignoring the sign, the corresponding term is 
\begin{equation}
X_{14}X_{29}X_{31}X_{47}X_{55}X_{66}X_{78}X_{83}X_{92}X_{11}X_{29}X_{38}X_{47}X_{56}X_{63}X_{74}X_{85}X_{92}.
\end{equation}
However, by symmetry, this is equal to
\begin{equation}
X_{11}X_{13}X_{14}X_{29}^4X_{36}X_{38}^2X_{47}^3X_{55}X_{56}X_{58}X_{66}X_{78}.
\end{equation}
Thus, we have that $w(t) = m_1^9 m_2 m_3 m_4$.

\subsection{Multi-graphs}\label{sec:multigraphs}
As observed by Lv in his Master's thesis \cite{zelin2023moments}, we can represent tables $t \in F_{k,n}$ by multi-graphs with colored edges. Using this observation, Lv \cite{zelin2023moments} rederived $F_2(t)$ for the case where $m_1=0$.

\begin{definition}
Let $\mathcal{S}_n$ be the set of all directed graphs on $n$ vertices (where loops are allowed) where each vertex has indegree and outdegree $1$. In other words, $\mathcal{S}_n$ is the set of graphs on $n$ vertices which are disjoint unions of cycles. Given a permutation $\pi \in S_n$, we define the corresponding graph $G_{\pi} \in \mathcal{S}_n$ to be the graph with vertices $V(G_{\pi}) = [n]$ and edges $E(G_{\pi}) = \{(i,\pi(i)): i \in [n]\}$.
\end{definition}
\begin{definition}
We define $\mathcal{F}_{k,n}$ to be the set of all $k$-tuples of graphs in $\mathcal{S}_n$. Given $t \in F_{k,n}$, we take the corresponding element $\hat{t} \in \mathcal{F}_{k,n}$ to be $\hat{t} = (G_{\pi_1},\ldots,G_{\pi_k})$ where $\pi_1,\ldots,\pi_k$ are the permutations corresponding to the rows of $t$. Visually, we represent $\hat{t}$ as a multi-graph with edges $E(\hat{t}) = \cup_{j=1}^{k}{E(G_{\pi_j})}$ where the edges in $E(G_{\pi_j})$ have color $j$. Since we only consider $k = 1$ and $k = 2$ in this paper, we represent the edges in $G_{\pi_1}$ with solid lines and we represent the edges in $G_{\pi_2}$ with dashed lines.
\end{definition}
It is not hard to check that there is a bijection between $F_{k,n}$ and $\mathcal{F}_{k,n}$. This bijection allows us to work with $\mathcal{F}_{k,n}$ rather than $F_{k,n}$.
\begin{definition}
Given $t \in F_{k,n}$,
\begin{enumerate}
\item We define $\sgn(\hat{t}) = \sgn(t)$. Note that 
\[
\sgn(\hat{t}) = \prod_{j=1}^{k}{(-1)^{C_E(\pi_j)}} = \prod_{j=1}^{k}{(-1)^{n - C(\pi_j)}} = (-1)^{kn - C(\hat{t})}
\]
where $C_{E}(\pi_j)$ is the number of cycles of $\pi_j$ of even length, $C(\pi_j)$ is the total number of cycles of $\pi_j$, and $C(\hat{t})$ is the total number of cycles in $\hat{t}$ such that all edges have the same color.
\item We define $w(\hat{t}) = w(t) = \Exx\left[\prod_{(i,j) \in E(\hat{t})}{X_{ij}}\right]$. Equivalently, $w(\hat{t})$ is the product of $m_{e_{ij}}$ over all pairs of vertices $i \geq j \in [n]$ where $e_{ij}$ is the number of edges between $i$ and $j$ (in both directions).
\end{enumerate}
\end{definition}
\begin{example}
The multi-graph $\hat{t}$ for the table $t$ from Figure \ref{fig:tabsymminim} is as follows.
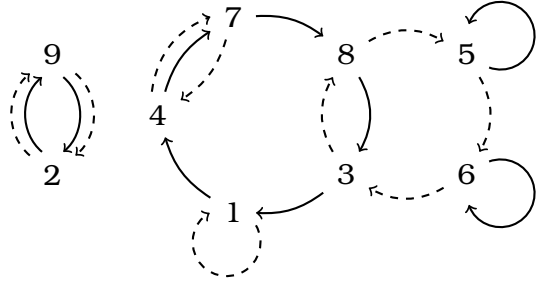
\begin{figure}[htb!]
\centering
\begin{tikzpicture}[scale = 1.0]
\node[vertex] (1) at (2.1,-1.3) {$1$};
\node[vertex] (4) at (1.1,0) {$4$}; 
\node[vertex] (7) at (2.1,1.3) {$7$};
\node[vertex] (8) at (3.6,0.8) {$8$};
\node[vertex] (3) at (3.6,-0.8) {$3$};
\node[vertex] (9) at (-0.3,0.8) {$9$};
\node[vertex] (2) at (-0.3,-0.8) {$2$};
\node[vertex] (5) at (5.2,0.8) {$5$};
\node[vertex] (6) at (5.2,-0.8) {$6$};
\draw[->,dashed] (9) to [bend left=50] (2);
\draw[->] (9.-65) to [bend left=50] (2.65);
\draw[->,dashed] (2) to [bend left=50] (9);
\draw[->] (2.115) to [bend left=50] (9.-115);
\draw[->] (4) to [bend right=-20] (7);
\draw[->] (7) to [bend right=-20] (8);
\draw[->] (8) to [bend right=-30] (3);
\draw[->] (3) to [bend right=-20] (1);
\draw[->] (1) to [bend right=-20] (4);
\draw[->,dashed] (7) to [bend right=-20] (4);
\draw[->,dashed] (4.105) to [bend right=-30] (7.180);
\draw[->,dashed] (8) to [bend right=-30] (5);
\draw[->,dashed] (5) to [bend right=-30] (6);
\draw[->,dashed] (6) to [bend right=-30] (3);
\draw[->,dashed] (3) to [bend right=-30] (8);
\draw[->] (5.-30) arc (-110:160:0.45);
\draw[->] (6.30) arc (110:-160:0.45);
\draw[->,dashed] (1.-30) arc (30:-240:0.45);
\end{tikzpicture}
\caption{A multi-graph $\hat{t} \in \mathcal{F}_{2,9}$ with weight $w(\hat{t}) = m_1^9 m_2 m_3 m_4$} \label{fig:multiexample}
\end{figure}
\end{example}
\begin{proposition} For any distribution of $X_{ij}$, we have
\begin{equation}
f_k(n) = \sum_{\hat{t} \in \mathcal{F}_{k,n}} \sgn(\hat{t})w(\hat{t}).
\end{equation}
and 
\begin{equation}
F_k(t) = \sum_{n=0}^{\infty}{\sum_{\hat{t} \in \mathcal{F}_{k,n}}{\frac{\sgn(\hat{t})w(\hat{t}){t^n}}{n!}}}.
\end{equation}
\end{proposition}
To analyze the exponential generating function $F_k(t)$, it is convenient to consider $\mathcal{F}_{k,n}$ for all $n$ simultaneously so we make the following definition. 
\begin{definition}
We define $\mathcal{F}_{k} = \bigcup_{n=0}^{\infty}{\mathcal{F}_{k,n}}$.
\end{definition}
\subsubsection{The first moment of the determinant of a random symmetric matrix}
As an example, we find $F_1(t)$ (and $f_1(n)$). Here, the structure of all graphs in $\mathcal{F}_{1}$ is simple (a permutation graph consists of disjoint cycles). See Figure \ref{fig:multi1example} for an example of such a graph.
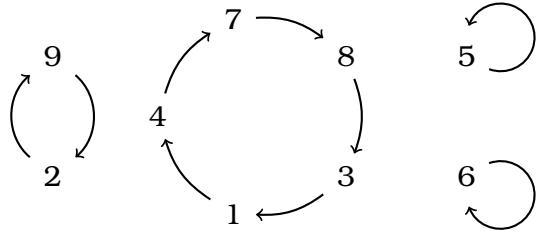
\begin{figure}[htb!]
\centering
\begin{tikzpicture}[scale = 1.0]
\node[vertex] (1) at (2.1,-1.3) {$1$};
\node[vertex] (4) at (1.1,0) {$4$}; 
\node[vertex] (7) at (2.1,1.3) {$7$};
\node[vertex] (8) at (3.6,0.8) {$8$};
\node[vertex] (3) at (3.6,-0.8) {$3$};
\node[vertex] (9) at (-0.3,0.8) {$9$};
\node[vertex] (2) at (-0.3,-0.8) {$2$};
\node[vertex] (5) at (5.2,0.8) {$5$};
\node[vertex] (6) at (5.2,-0.8) {$6$};
\draw[->] (9) to [bend right=-50] (2);
\draw[->] (2) to [bend left=50] (9);
\draw[->] (4) to [bend right=-20] (7);
\draw[->] (7) to [bend right=-20] (8);
\draw[->] (8) to [bend right=-20] (3);
\draw[->] (3) to [bend right=-20] (1);
\draw[->] (1) to [bend right=-20] (4);
\draw[->] (5.-30) arc (-110:160:0.45);
\draw[->] (6.30) arc (110:-160:0.45);
\end{tikzpicture}
\caption{A graph $\hat{t} \in \mathcal{F}_{1,9}$ with weight $w(\hat{t}) = m_1^7 m_2$} \label{fig:multi1example}
\end{figure}

\noindent
We analyze $F_1(t)$ by decomposing the graphs in $\mathcal{F}_1$ into a few different types of structures. To each structure, we assign its exponential generating function (EGF). The overall $F_1(t)$ (which is also an EGF) is then found by multiplying all of these partial EGF's. For $F_1(t)$, there are three types of structures:
\begin{itemize}
    \item Fixed points (cycles of length one) with weight $W = {m_1}t$.
    \item Mussels (cycles of length two) with weight $W = -{m_2}t^2$.
    \item Loops (cycles of length at least $3$) where a loop of length $l$ has weight $W = -m_1^l(-t)^l$.
\end{itemize}
\begin{remark}
We use $W$ rather than $w$ for these weights as we are including the factors of $t$ and the signs in these weights.
\end{remark}
\begin{remark}
Note that in order to verify the results by computer for small matrices, it is better 
to compute the logarithm of the EGF tather than the EGF as this makes the contribution of each structure additive rather than multiplicative. This allows us to heuristically identify missing structures we have not taken into account yet.
\end{remark}
\paragraph{Fixed points} Let $\mathcal{X}_1$ be the set of fixed points in $\mathcal{F}_1$. Then
\begin{equation}
    \mathcal{X}_1 = \textsc{Set} \bigg{(}\!\!\!
\begin{tikzpicture}[baseline=-0.4ex, scale = 0.8]
\node[vertex] (1) at (0,0) {$1$};
\draw[->] (1.-30) arc (-110:160:0.45);
\end{tikzpicture}
\bigg{)}.
\end{equation}
For the EGF, this gives a factor of $\exp(m_1 t)$.

\paragraph{Mussels} Next, let $\mathcal{M}_1$ be the set of mussels in $\mathcal{F}_1$. Then
\begin{equation}
    \mathcal{M}_1 = \textsc{Set}\Bigg{(}
\begin{tikzpicture}[baseline=-0.65ex, scale = 0.8]
\node[vertex] (2) at (0,0.8) {$2$};
\node[vertex] (1) at (0,-0.8) {$1$};
\draw[->] (2) to [bend left=50] (1);
\draw[->] (1) to [bend left=50] (2);
\end{tikzpicture}
\Bigg{)}.
\end{equation}
The EGF of the inner factor is $-m_2 t^2/2!$ so the EGF of $\mathcal{M}_1$ is $\exp(-\tfrac12 m_2 t^2)$.

\paragraph{Loops}
Finally, let $\mathcal{L}_1$ be the set of loops in $\mathcal{F}_1$. Then
\begin{equation}
    \mathcal{L}_1 = \textsc{Set} \bigg{(} \textsc{Cyc}_{\geq 3} \bigg{(}
\begin{tikzpicture}[baseline=0.6ex, scale = 0.8]
\node[vertex] (1) at (0,0) {$1$};
\node[vertex] (2) at (1.3,1) {};
\draw[->] (1) to [bend left=50] (2);
\end{tikzpicture}\!\!\!\!\!\!
\bigg{)}\bigg{)}.
\end{equation}
The EGF for cycles of length $n \geq 3$ is $-m_1t + \frac{m_1^2 t^2}{2} - \ln\left(\frac{1}{1+m_1 t}\right)$ so the EGF of $\mathcal{L}_1$ is equal to 
\begin{equation*}
\exp\left(-m_1t + \frac{m_1^2 t^2}{2} - \ln\left(\frac{1}{1+m_1 t}\right)\right) = (1+m_1 t) e^{-m_1 t + \frac12 m_1^2 t^2}.
\end{equation*}

\paragraph{Conclusion}
Since $\mathcal{F}_1 = \mathcal{X}_1\star\mathcal{M}_1\star\mathcal{L}_1$, by multiplying the EGFs of the structures we obtain that
\begin{equation}\label{Eq:F1}
F_1(t) = e^{m_1 t} e^{-\frac12 m_2 t^2} (1+m_1 t) e^{-m_1 t + \frac12 m_1^2 t^2} = (1+m_1 t) e^{\frac12 (m_1^2-m_2) t^2}.
\end{equation}
\subsection{Techniques for handling nonzero \texorpdfstring{$m_1$}{m1}}\label{sec:markedmultigraphs}
To analyze the second moment of the determinant of $X$ when $m_1 \neq 0$, it is more convenient to write $X = Y + {m_1}J$ and carry out our analysis in terms of the entries of $Y$.
\begin{definition}
We take $Y = X - {m_1}J$ so we have that $Y_{ij} = X_{ij} - m_1$. We define $\mu_q = \Exx [Y_{ij}^q]$ to be the $q$th moment of the entries of $Y$. Note that $\mu_1 = 0$ and $\mu_2 = {m_2} - {m_1}^2$.
\end{definition}
\begin{proposition}
\label{MainProp}\label{prop:MatLem}
\begin{equation}
\det X = \!\!\!\sum_{\pi \in F_n} \sgn(\pi) \prod_{i \in \left[n\right]} Y_{i\pi(i)} \, +\, m_1\sum_{j \in [n]} \sum_{\pi \in F_n} \prod_{i \in \left[n\right]\setminus \{j\}} Y_{i \pi(i)}.
\end{equation}
\end{proposition}
\begin{proof}
Using the fact that $\det X = \sum_{\pi \in S_n}  \sgn(\pi) \prod_{i \in \left[n\right]} X_{i \pi(i)}$ and multiplying everything out,
\begin{equation*}
\det X \! =\! \sum_{\pi \in F_n}  \sgn (\pi) \prod_{i \in \left[n\right]} X_{i \pi(i)} \!=\! \sum_{\pi \in S_n}  \sgn (\pi) \prod_{i \in \left[n\right]} (Y_{i \pi(i)}+m_1) \!=\! \sum_{\pi \in S_n}  \sgn (\pi)\sum_{A \subseteq \left[n\right]} \prod_{i \in \left[n\right]\setminus A} m_1^{|A|} Y_{i \pi(i)},
\end{equation*}
We claim that the terms with $|A| \geq 2$ vanish. To see this, given an $A$ such that $|A| \geq 2$, let $a_1$ and $a_2$ be the first two elements of $A$. For each permutation $\pi \in S_n$, let $\operatorname{swap}_{\pi}$ be the permutation such that $\operatorname{swap}_{\pi}(a_1) = \pi(a_2)$, $\pi(a_2) = \operatorname{swap}_{\pi}(a_1)$, and $\operatorname{swap}_{\pi}(i) = \pi(i)$ for all $i \in [n] \setminus \{a_1,a_2\}$. Since $\sgn(\operatorname{swap}_{\pi}) = -\sgn(\pi)$ for all $\pi \in S_n$ and $\operatorname{swap}_{\pi}(i) = \pi(i)$ for all $i \in [n] \setminus A$, we have that 
\begin{equation*}
\sum_{\pi \in S_n: \pi(a_1) < \pi(a_2)}{\left(\sgn (\pi) \left(\prod_{i \in \left[n\right]\setminus A} {m_1^{|A|} Y_{i \pi(i)}}\right) + \sgn (\operatorname{swap}_{\pi}) \left(\prod_{i \in \left[n\right]\setminus A} {m_1^{|A|} Y_{i \operatorname{swap}_{\pi}(i)}}\right)\right)} = 0.
\end{equation*}
\end{proof}
\begin{remark}
Proposition \ref{prop:MatLem} is a special case of the Matrix Determinant Lemma which says that $\det(A + uv^T) = (1 + v^TA^{-1}u)\det(A)$. We gave an alternative direct proof as it helps motivate the definition of marked permutations below.
\end{remark}
\begin{definition}
We say $\sigma$ is a \emph{marked permutation} if it is either in $S_n$ or it is formed from some $\pi \in S_n$ by setting $\sigma(i) = \times$ for some $i \in [n]$ and $\sigma(j) = \pi(j)$ for all $j \in [n] \setminus \{i\}$. In this case, we think of $\times$ as a mark and we call $i$ the marked element of $\sigma$. 

We define $\sgn(\sigma) = \sgn (\pi)$ and we take $Y^\times_{i\sigma(i)} = m_1$ if $i$ is marked and $Y^\times_{i\sigma(i)} = Y_{i\pi(i)}$ otherwise. We write $S^\times_n$ for the set of all marked permutations.
\end{definition}
Restating Proposition \ref{MainProp} in terms of marked permutations, we have the following equation.
\begin{proposition} 
\begin{equation*}
\det X = \sum_{\sigma \in S_n^\times} \sgn(\sigma) \prod_{i=1}^n Y_{i \sigma(i)}^\times.
\end{equation*}
\end{proposition}
We can easily modify tables and multi-graphs to take these marks into account.
\begin{definition}\label{Def:GfromF}
We define $G^{\times}_{k,n}$ to be the set of marked tables formed from $F_{k,n}$ by marking at most one element in each row. Note that $G^{\times}_{k,n}$ includes the tables where no elements are marked.

Given a table $t \in G^{\times}_{k,n}$ we define its weight to be $w(t) = \Exx \left[\prod_{j=1}^k \prod_{i=1}^n Y^{\times}_{i\sigma_j(i)}\right]$ where $\sigma_j$ is the marked permutation corresponding to row $j$ of $t$ (i.e., $\sigma_j(i)$ is the element in row $j$ and column $i$ of $t$ if this element is unmarked and $\times$ if this element is marked.).
\end{definition}
\begin{definition}
Given a marked table $t \in G^\times_{k,n}$, we construct a marked multi-graph $\hat{t}$ by first taking the multi-graph corresponding to $t$ when we ignore the marks. We then mark the edges (if any) corresponding to marked elements in $t$. Note that at most one edge of each color is marked.

As before, we define $w(\hat{t}) = w(t)$ and $\sgn (\hat{t}) = \sgn (t)$. We define $\mathcal{G}^\times_{k,n}$ to be the set of all marked multi-graphs corresponding to a marked table $t \in G^\times_{k,n}$.
\end{definition}
\begin{definition}
Given $t \in G^{\times}_{k,n}$, we say that $t$ and $\hat{t}$ are trivial if there is an $i \in [n]$ such that there is exactly one unmarked loop incident to $i$ in $\hat{t}$ or there exist $i,j \in [n]$ such that there is exactly one unmarked edge (in either direction) between $i$ and $j$ in $\hat{t}$. Equivalently, $t$ and $\hat{t}$ are  trivial if there exists an $i \in [n]$ such that $i$ appears exactly once in column $i$ of $t$ or there exist $i,j \in [n]$ such that the number of times $j$ appears in column $i$ of $t$ plus the number of times $i$ appears in column $j$ of $t$ is exactly $1$.

Otherwise, we say that $t$ and $\hat{t}$ are non-trivial.
\end{definition}
We observe that trivial tables/multi-graphs have weight $0$ so it is sufficient to restrict our attention to non-trivial tables/multi-graphs.
\begin{proposition}
If $t$ and $\hat{t}$ are trivial then $w(t) = w(\hat{t}) = 0$.
\end{proposition}
We will split our analysis into cases depending on the number of marks in the table/multi-graph.
\begin{definition}
We define $G^{r}_{k,n}$ to be the set of tables in $G^{\times}_{k,n}$ which have exactly $r$ marks. Similarly, we define $\mathcal{G}^{r}_{k,n}$ to be the set of multi-graphs in $\mathcal{G}^{r}_{k,n}$ which have exactly $r$ marked edges. Note that we can have $r = 0$, in that case, we often omit the superscript.
\end{definition}
\begin{definition}
We define
\begin{equation}
g^r_k(n) = \sum_{\hat{t} \in \mathcal{G}^r_{k,n}} w(\hat{t}) \sgn (t)  \qquad \text{and} \qquad G^r_k(t) = \sum_{n=0}^\infty \frac{t^n}{n!} g^r_k(n).
\end{equation}
\end{definition}
\begin{proposition}\label{prop:hatFfromhatG}
For any distribution of $X_{ij}$,
\begin{equation}
f_k(n) = \sum_{r=0}^k g^r_k(n) \qquad \text{and} \qquad F_k(t) = \sum_{r=0}^k G_k^r(t).
\end{equation}
\end{proposition}
In order to apply the techniques of Flajolet and Sedgewick \cite{flajolet2009analytic}, it is convenient to absorb $\sgn(\hat{t})$ and $t^n$ into the weight of $\hat{t}$.
\begin{definition}
We define $\mathcal{G}^\times_k$ to be the (disjoint) union $\bigsqcup_{n=0}^\infty \mathcal{G}^\times_{k,n}$. For a given graph $\hat{g} \in \mathcal{G}^\times_k$, such that $\hat{g} \in \mathcal{G}^\times_{k,n}$ for some $n$, we define the corresponding weight $W(\hat{g}) =  w(\hat{g}) \sgn(\hat{g}) t^n$, where $\sgn(\hat{g}) = (-1)^{kn} (-1)^{C(\hat{g})}$ and we define the order $|\hat{g}|$ of $\hat{g}$ to be $|\hat{g}| = n$. Similarly, we define $\mathcal{G}^r_k = \bigsqcup_{n=0}^\infty \mathcal{G}^r_{k,n}$.
\end{definition}
\begin{proposition}
For any distribution,
\begin{equation}
F_k(t) = \sum_{\hat{g} \in \mathcal{G}^\times_k} \frac{W(\hat{g})}{|\hat{g}|!}.
\end{equation}
\end{proposition}

\begin{remark}\label{Rem:GfromF}
Note that $\mathcal{G}_k = \mathcal{G}_k^0$ tables (with no marks) and $\mathcal{F}_n$ tables are equivalent when $m_1 = 0$. More precisely, $f_k(n)|_{m_1=0} = g_k(n)|_{\mu_q\to m_q}$ and thus $F_k(t)|_{m_1=0} = G_k(t)|_{\mu_q\to m_q}$.
\end{remark}

\begin{example}\label{ex:G2tabletimes}
An example of nontrivial table $t \in G^\times_{2,9}$ and its corresponding multi-graph.
\begin{figure}[H]
\centering
\begin{minipage}{.48\textwidth}
        \centering
    \begin{tabular}{|c|c|c|c|c|c|c|c|c|}
    \hline
        1 & 2 & 3 & 4 & 5 & 6 & 7 & 8 & 9\\
    \hdashline
        3 & 7 & 1 & $\!\!\times\!\!$ & 5 & 2 & 6 & 9 & 8\\
        3 & 6 & 1 & 9 & 5 & 7 & 2 & $\!\!\times\!\!$ & 4\\
    \hline
    \end{tabular}
\caption{$t \in  G^\times_{2,9}$ with weight $w(t) = \mu_2^6\mu_4m_1^2$,}
    \end{minipage}%
\hfill
\begin{minipage}{.48\textwidth}
        \centering
\begin{tikzpicture}[scale = 0.85]
\node[vertex] (1) at (-2.5,1) {$1$};
\node[vertex] (3) at (-2.5,-1) {$3$};
\node[vertex] (2) at (-1,0.8) {$2$};
\node[vertex] (6) at (1,0.8) {$6$}; 
\node[vertex] (7) at (0,-1) {$7$};
\node[vertex] (5) at (5,0) {$5$};
\node[vertex] (4) at (3,1) {$4$};
\node[vertex] (9) at (2,0) {$9$};
\node[vertex] (8) at (3,-1) {$8$};
\draw[->,dashed] (1) to [bend left=50] (3);
\draw[->] (1.-65) to [bend left=50] (3.65);
\draw[->,dashed] (3) to [bend left=50] (1);
\draw[->] (3.115) to [bend left=50] (1.-115);
\draw[->,dashed] (2) to [bend left=35] (6);
\draw[->] (6.182) to [bend left=-40] (2.-002);
\draw[->,dashed] (6) to [bend left=40] (7);
\draw[->] (7.55) to [bend left=-40] (6.250);
\draw[->,dashed] (7) to [bend left=40] (2);
\draw[->] (2.290) to [bend left=-40] (7.130);
\draw[->] (5.-60) arc (180+45+20:540-45-15:0.45);
\draw[->,dashed] (5.-90) arc (180+45:540-45:0.6);
        \draw[->] (4.-70+0) arc (-130+0:125+0:0.45) node[fill=white,midway,sloped]{$\!\!\times\!\!$};
        \draw[->,dashed] (4) to [bend left=20] (9);
        \draw[->,dashed] (9) to [bend left=20] (4);
        \draw[->] (9) to [bend left=20] (8);
        \draw[->] (8) to [bend left=20] (9);
        \draw[->,dashed] (8.-70) arc (-130:125:0.45) node[fill=white,midway,sloped]{$\!\!\times\!\!$};
\end{tikzpicture}
\caption{The corresponding $\hat{t} \in \mathcal{G}^\times_{2,9}$}
\label{fig:G2tabletimes}
    \end{minipage}%
\end{figure}
\end{example}

\subsection{Rederivation of the first symmetric moment}
To demonstrate the technique of using marked tables, we rederive $F_1(t)$.

\subsubsection{Zero marks}
Consider the set of all nontrivial $1$-graphs in $\mathcal{G}_1 = \mathcal{G}^0_1$. This set includes only the permutations which decompose into mussels ($2$-cycles). Thus,
\begin{equation}
        \mathcal{G}_1 = \mathcal{G}_1^0 = \textsc{Set}\Bigg{(}
\begin{tikzpicture}[baseline=-0.65ex, scale = 0.8]
\node[vertex] (2) at (0,0.8) {$2$};
\node[vertex] (1) at (0,-0.8) {$1$};
\draw[->] (2) to [bend left=50] (1);
\draw[->] (1) to [bend left=50] (2);
\end{tikzpicture}
\Bigg{)}.
\end{equation}
In terms of the EGF,
\begin{equation}
G^0_1(t) = e^{-\frac12 \mu_2 t^2}.
\end{equation}

\subsubsection{One mark}
It turns out there is only one nontrivial structure which can have one mark:
\begin{itemize}
    \item A marked fixed point (marked $1$-cycle) with weight $W=m_1 t$
\end{itemize}
Note that there are no marked mussels as they would have weight equal to zero.

\paragraph{Marked fixed points} As there is only one mark per multi-graph $\hat{t}$ in $\mathcal{G}^1_1$, there must be exactly one structure containing it. The only nontrivial structure of $\mathcal{G}_1^1$ is precisely the marked fixed point $\mathcal{X}_1$, whose EGF is $m_1 t$.

\paragraph{Conclusion}
Since other structures of $\mathcal{G}_1^1$ are taken from $\mathcal{G}_1$, we have $\mathcal{G}_1^1 = \mathcal{X}_1 \star \mathcal{G}_1$. Hence, in terms of the EGFs, $G^1_1(t) = m_1t\,G_1(t)$.

\subsubsection{Full generating function}
By Proposition \ref{prop:hatFfromhatG},
\begin{equation}
F_1(t) = G^0_1(t)+G^1_1(t) = (1+m_1 t)G_1(t) = (1+m_1 t) e^{-\frac12 \mu_2 t^2} = (1+m_1 t) e^{\frac12 (m_1^2-m_2) t^2},
\end{equation}
which matches Equation \eqref{Eq:F1}.

\subsection{Second symmetric moment}
Next, we move to the case of $k = 2$. To derive the full generating function $F_2(t)$, we must analyze all of the possible structures of nontrivial graphs in $\mathcal{G}^\times_2$.

\subsubsection{Zero marks}
\begin{proposition}
For any distribution of $X_{ij}$ with $m_1 = 0$,
\begin{equation}
F_2(t) = \frac{\exp\left(\frac{(m_4-3m_2^2) t^2}{2}-m_2 t\right)}{\left(1-m_2 t\right){}^2 \sqrt{1-m_2^2 t^2}}
\end{equation}
\end{proposition}

\begin{remark}
For the special case when $m_2 = 1$, we obtain
\begin{equation}
F_2(t)=\frac{\exp\left(\frac{(m_4-3) t^2}{2}-t\right)}{(1-t)^2 \sqrt{1-t^2}}
\end{equation}
\end{remark}
\begin{proof}
First, in the spirit of Remark \ref{Rem:GfromF}, we consider the structure of nontrivial graphs in $\mathcal{G}_{2}$. An example of such graph is in Figure \ref{fig:G2table}.

\begin{example}\label{ex:G2table}
An example of a nontrivial table $t \in G^0_{2,10}$ and its corresponding multi-graph.
\begin{figure}[H]
\centering
\begin{minipage}{.48\textwidth}
        \centering
    \begin{tabular}{|c|c|c|c|c|c|c|c|c|c|}
    \hline
        1 & 2 & 3 & 4 & 5 & 6 & 7 & 8 & 9 & \!10\!\\
    \hdashline
        3 & 7 & 1 & 9 & 8 & 2 & 6 & 5 & 4 & \!10\!\\
        3 & 6 & 1 & 5 & 4 & 7 & 2 & 9 & 8 & \!10\!\\
    \hline
    \end{tabular}
\caption{$t \in  G_{2,10}$ with weight $w(t) = \mu_2^8\mu_4$,}
    \end{minipage}%
\hfill
\begin{minipage}{.48\textwidth}
        \centering
\begin{tikzpicture}[scale = 0.85]
\node[vertex] (1) at (-2.5,1) {$1$};
\node[vertex] (3) at (-2.5,-1) {$3$};
\node[vertex] (2) at (-1,0.8) {$2$};
\node[vertex] (6) at (1,0.8) {$6$}; 
\node[vertex] (7) at (0,-1) {$7$};
\node[vertex] (10) at (5,0) {$10$};
\node[vertex] (4) at (2,0) {$4$};
\node[vertex] (9) at (3,1) {$9$};
\node[vertex] (8) at (4,0) {$8$};
\node[vertex] (5) at (3,-1) {$5$};
\draw[->,dashed] (1) to [bend left=50] (3);
\draw[->] (1.-65) to [bend left=50] (3.65);
\draw[->,dashed] (3) to [bend left=50] (1);
\draw[->] (3.115) to [bend left=50] (1.-115);
\draw[->,dashed] (2) to [bend left=35] (6);
\draw[->] (6.182) to [bend left=-40] (2.-002);
\draw[->,dashed] (6) to [bend left=40] (7);
\draw[->] (7.55) to [bend left=-40] (6.250);
\draw[->,dashed] (7) to [bend left=40] (2);
\draw[->] (2.290) to [bend left=-40] (7.130);
\draw[->,dashed] (9) to [bend left=20] (8);
\draw[->,dashed] (8) to [bend left=20] (9);
\draw[->] (8) to [bend left=20] (5);
\draw[->] (5) to [bend left=20] (8);
\draw[->,dashed] (5) to [bend left=20] (4);
\draw[->,dashed] (4) to [bend left=20] (5);
\draw[->] (4) to [bend left=20] (9);
\draw[->] (9) to [bend left=20] (4);
\draw[->] (10.-60) arc (180+45+20:540-45-15:0.45);
\draw[->,dashed] (10.-90) arc (180+45:540-45:0.6);
\end{tikzpicture}
\caption{A corresponding $\hat{t} \in \mathcal{G}_{2,10}$}
\label{fig:G2table}
    \end{minipage}%
\end{figure}
\end{example}

This example was purposely created to show all types of structures a graph $\hat{t}$ from $\mathcal{G}_2$ can have. For $\mathcal{G}_{2}$, we have the following structures.
\begin{itemize}
    \item Fixed points (two $1$-cycles on the same element) with weight $W = \mu_2 t$.
    \item Mussels (two $2$-cycles on two elements) with weight $W = \mu_4 t^2$.
    \item Loops (double cycles on $n\geq 3$ vertices) with weight $W = \mu_2^n t^n$.
    \item Necklaces (closed chains of alternating dashed/undashed $2$-cycles on $n$ vertices where $n \geq 4$ is even) with weight $W = \mu_2^n t^n$.
\end{itemize}
Representatives of each type are also shown in Table \ref{tab:typesG20}.
\begin{table}[H]
    \centering
\begin{tabular}{cccc}
        \begin{tikzpicture}[baseline = -1.3ex,scale = 0.85]
        \node[vertex] (1) at (-2.5,1) {$1$};
        \draw[->] (1.-60) arc (180+45+20:540-45-15:0.45);
        \draw[->,dashed] (1.-90) arc (180+45:540-45:0.6);
        \end{tikzpicture}
     &
        \begin{tikzpicture}[scale = 0.85]
        \node[vertex] (1) at (-2.5,1) {$1$};
        \node[vertex] (2) at (-2.5,-1) {$2$};
        \draw[->] (1.-65) to [bend left=50] (2.65);
        \draw[->,dashed] (2) to [bend left=50] (1);
        \draw[->] (2.115) to [bend left=50] (1.-115);
        \draw[->,dashed] (1) to [bend left=50] (2);
        \end{tikzpicture}
     &
        \begin{tikzpicture}[scale = 0.85]
        \node[vertex] (1) at (-1,0.8) {$1$};
        \node[vertex] (2) at (1,0.8) {$2$}; 
        \node[vertex] (3) at (0,-1) {$3$};
        \draw[->] (2.182) to [bend left=-40] (1.-002);
        \draw[->,dashed] (1) to [bend left=35] (2);
        \draw[->,dashed] (2) to [bend left=40] (3);
        \draw[->] (3.55) to [bend left=-40] (2.250);
        \draw[->,dashed] (3) to [bend left=40] (1);
        \draw[->] (1.290) to [bend left=-40] (3.130);
        \end{tikzpicture}      
     &
        \begin{tikzpicture}[scale = 0.85]
        \node[vertex] (4) at (3,1) {$4$};
        \node[vertex] (3) at (4,0) {$3$};
        \node[vertex] (2) at (3,-1) {$2$};
        \node[vertex] (1) at (2,0) {$1$};
        \draw[->,dashed] (4) to [bend left=20] (3);
        \draw[->,dashed] (3) to [bend left=20] (4);
        \draw[->] (3) to [bend left=20] (2);
        \draw[->] (2) to [bend left=20] (3);
        \draw[->,dashed] (2) to [bend left=20] (1);
        \draw[->,dashed] (1) to [bend left=20] (2);
        \draw[->] (1) to [bend left=20] (4);
        \draw[->] (4) to [bend left=20] (1);
        \end{tikzpicture} 
     \\
     Fixed point & Mussel & Loop & Necklace
        \\
        $\mathcal{X}_2$
        & $\mathcal{M}_2$
        & $\mathcal{L}_2$
        & $\mathcal{N}_2$
\end{tabular}
    \caption{Types of structures in $\mathcal{G}_2$}
    \label{tab:typesG20}
\end{table}

\paragraph{Fixed points} Let $\mathcal{X}_2$ be the set of fixed points in $\mathcal{G}_2$. Then
\begin{equation}
    \mathcal{X}_2 = \textsc{Set} \bigg{(}\!\!\!
        \begin{tikzpicture}[baseline=3.7ex,scale = 0.85]
        \node[vertex] (1) at (-2.5,1) {$1$};
        \draw[->] (1.-60) arc (180+45+20:540-45-15:0.45);
        \draw[->,dashed] (1.-90) arc (180+45:540-45:0.6);
        \end{tikzpicture}
\bigg{)}.
\end{equation}
In terms of the EGF, this gives a factor of $\exp(\mu_2 t)$.

\paragraph{Mussels}
Next, let $\mathcal{M}_2$ be the set of mussels in $\mathcal{G}_2$. Then
\begin{equation}
    \mathcal{M}_2 = \textsc{Set}\Bigg{(}
        \begin{tikzpicture}[baseline=-0.65ex,scale = 0.85]
        \node[vertex] (1) at (-2.5,1) {$1$};
        \node[vertex] (2) at (-2.5,-1) {$2$};
        \draw[->] (1.-65) to [bend left=50] (2.65);
        \draw[->,dashed] (2) to [bend left=50] (1);
        \draw[->] (2.115) to [bend left=50] (1.-115);
        \draw[->,dashed] (1) to [bend left=50] (2);
        \end{tikzpicture}
\Bigg{)}.
\end{equation}
The EGF of the inner factor is $\mu_4 t^2/2!$ so the EGF of $\mathcal{M}_2$ is $\exp(\tfrac12 \mu_4 t^2)$.

\paragraph{Loops} There are two ways how we can make a loop. Either the undashed and dashed cycles have the same orientation, or they have opposite orientations. Letting $\mathcal{L}_2$ be the set of loops in $\mathcal{G}_2$, we have
\begin{equation}
    \mathcal{L}_2 = \textsc{Set} \bigg{(} \textsc{Cyc}_{\geq 3} \bigg{(} \oaloodup \!\!\!\!\!\!
\bigg{)}\bigg{)} \star
\textsc{Set} \bigg{(} \textsc{Cyc}_{\geq 3} \bigg{(} \oalooduo \!\!\!\!\!\!
\bigg{)}\bigg{)}.
\end{equation}
The corresponding EGF of $\mathcal{L}_2$ is equal to
\begin{equation}
\left[\exp\left(-\mu_2t - \frac{\mu_2^2 t^2}{2} + \ln\left(\frac{1}{1-\mu_2 t}\right)\right)\right]^2 = \frac{e^{-2\mu_2 t -\mu_2^2 t^2}}{(1-\mu_2 t)^2}.
\end{equation}

\paragraph{Necklaces} Finally, the set of necklaces $\mathcal{N}_2$ is formed as cycles of two alternating $2$-cycles joined together. Note that it depends on the order of the elements in consecutive $2$-cycles. However, every necklace is counted twice as the mirror image of a necklace is counted as the same necklace. Overall, we have that
\begin{equation}
    \mathcal{N}_2 = \textsc{Set} \Bigg{(} \frac12 \,  \textsc{Cyc}_{\geq 2} \Bigg{(}
\obchauab \!\!\!\!\!\! + \!\!\! \obchauba \!\!\! \Bigg{)}\Bigg{)}.
\end{equation}
The resulting EGF is
\begin{equation}
\exp\left(\frac12\left(-2\mu_2^2\frac{t^2}{2!}+\ln\left(\frac{1}{1-2\mu_2^2\frac{t^2}{2!}}\right)\right)\right) = \frac{e^{-\frac12 \mu_2^2 t^2}}{\sqrt{1-\mu_2^2 t^2}}.   
\end{equation}

\paragraph{Conclusion}
Since $\mathcal{G}_2 = \mathcal{X}_2 \star \mathcal{M}_2 \star \mathcal{L}_2 \star \mathcal{N}_2$, by multiplying the EGFs we obtain that
\begin{equation}
G_2(t)= e^{\mu_2 t} e^{\frac12 \mu_4 t^2} \frac{e^{-2\mu_2 t -\mu_2^2 t^2}}{(1-\mu_2 t)^2} \frac{e^{-\frac12 \mu_2^2 t^2}}{\sqrt{1-\mu_2^2 t^2}} = \frac{\exp\left(\frac{(\mu _4-3\mu_2^2) t^2}{2}-\mu _2 t\right)}{\left(1-\mu _2 t\right){}^2 \sqrt{1-\mu _2^2 t^2}}.
\end{equation}
\end{proof}

\subsubsection{One mark}
Note that all unmarked structures are also part of every graph $\mathcal{G}^r_2$ for $r \geq 1$. Thus, we only need to find the exponential generating function of connected structures having marks and add them up (as they are disjoint). It turns out that the only marked structure of each $\hat{t} \in \mathcal{G}_2^1$ is a mussel with exactly one mark ($1$-marked mussel).
\begin{example}\label{ex:G21table}
An example of a nontrivial table $t \in G^1_{2,9}$ and its corresponding nontrivial multi-graph $\hat{t} \in \mathcal{G}^1_{2,9}$.
\begin{figure}[H]
\centering
\begin{minipage}{.48\textwidth}
        \centering
    \begin{tabular}{|c|c|c|c|c|c|c|c|c|}
    \hline
        1 & 2 & 3 & 4 & 5 & 6 & 7 & 8 & 9 \\
    \hdashline
        3 & 7 & 1 & 9 & 8 & 2 & 6 & 5 & 4 \\
        3 & 6 & $\times$ & 5 & 4 & 7 & 2 & 9 & 8\\
    \hline
    \end{tabular}
\caption{$t \in  G_{2,9}$ with $w(t) = m_1\mu_2^7 \mu_3$,}
    \end{minipage}%
\hfill
\begin{minipage}{.48\textwidth}
        \centering
\begin{tikzpicture}[scale = 0.85]
\node[vertex] (1) at (-2.5,1) {$1$};
\node[vertex] (3) at (-2.5,-1) {$3$};
\node[vertex] (2) at (-1,0.8) {$2$};
\node[vertex] (6) at (1,0.8) {$6$}; 
\node[vertex] (7) at (0,-1) {$7$};
\node[vertex] (4) at (2,0) {$4$};
\node[vertex] (9) at (3,1) {$9$};
\node[vertex] (8) at (4,0) {$8$};
\node[vertex] (5) at (3,-1) {$5$};
\draw[->,dashed] (1) to [bend left=50] (3);
\draw[->] (1.-65) to [bend left=50] (3.65);
\draw[->,dashed] (3) to [bend left=50] node[fill=white,sloped]{$\!\!\!\times\!\!\!$} (1);
\draw[->] (3.115) to [bend left=50] (1.-115);
\draw[->,dashed] (2) to [bend left=35] (6);
\draw[->] (6.182) to [bend left=-40] (2.-002);
\draw[->,dashed] (6) to [bend left=40] (7);
\draw[->] (7.55) to [bend left=-40] (6.250);
\draw[->,dashed] (7) to [bend left=40] (2);
\draw[->] (2.290) to [bend left=-40] (7.130);
\draw[->,dashed] (9) to [bend left=20] (8);
\draw[->,dashed] (8) to [bend left=20] (9);
\draw[->] (8) to [bend left=20] (5);
\draw[->] (5) to [bend left=20] (8);
\draw[->,dashed] (5) to [bend left=20] (4);
\draw[->,dashed] (4) to [bend left=20] (5);
\draw[->] (4) to [bend left=20] (9);
\draw[->] (9) to [bend left=20] (4);
\end{tikzpicture}
\caption{The corresponding $\hat{t} \in \mathcal{G}_{2,10}$}
\label{fig:G21table}
    \end{minipage}%
\end{figure}
\end{example}
The EGF of a mussel with exactly one mark is equal to $\tfrac12 m_1\mu_3 t^2$. However, there are four places where we can put our mark so the full EGF is $2m_1\mu_3 t^2$ and
\begin{equation}
    G_2^1(t) = 2m_1\mu_3 t^2G_2(t)
\end{equation}
as we fill the remaining space with unmarked structures.

\subsubsection{Two marks}
The following Table \ref{tab:typesG22} shows representatives of all types of the remaining (connected) structures having exactly two marks. A structure with two marks can be
\begin{itemize}
    \item A marked loop (double cycle on $n\geq 2$ vertices with two marks) with weight $W = m_1^2\mu_2^{n-1} t^n$.
    \item A marked chain (alternating dashed/undashed $2$-cycle on an odd number $n$ of vertices, provided we attach two marked $1$-cycles to the endpoints) with weight $W = m_1^2\mu_2^{n-1} t^n$.
\end{itemize}
\begin{table}[H]
    \centering
\begin{tabular}{cc}
        \begin{tikzpicture}[scale = 0.85]
        \node[vertex] (1) at (-1,0.8) {$1$};
        \node[vertex] (2) at (1,0.8) {$2$}; 
        \node[vertex] (3) at (0,-1) {$3$};
        \draw[->] (2.182) to [bend left=-40] (1.-002);
        \draw[->,dashed] (1) to [bend left=35] node[fill=white,sloped]{$\!\!\!\times\!\!\!$} (2);
        \node[vertex] (x) at (0,1) {$\times$};
        \draw[->,dashed] (2) to [bend left=40] (3);
        \draw[->] (3.55) to [bend left=-40] (2.250);
        \draw[->,dashed] (3) to [bend left=40] (1);
        \draw[->] (1.290) to [bend left=-40] (3.130);
        \end{tikzpicture}    
    &
        \begin{tikzpicture}[scale = 0.85]
        \node[vertex] (1) at (0,0) {$1$};
        \node[vertex] (2) at (1.5,1) {$2$};
        \node[vertex] (3) at (3,2) {$3$};
        \draw[->] (1.-70+180) arc (-130+180:125+180:0.45) node[fill=white,midway,sloped]{$\!\!\times\!\!$};
        \draw[->,dashed] (1) to [bend left=20] (2);
        \draw[->,dashed] (2) to [bend left=20] (1);
        \draw[->] (2) to [bend left=20] (3);
        \draw[->] (3) to [bend left=20] (2);
        \draw[->,dashed] (3.-70) arc (-130:125:0.45) node[fill=white,midway,sloped]{$\!\!\times\!\!$};
        \end{tikzpicture}       
    \\
     $2$-Marked loop & $2$-Marked chain
        \\
        $\mathcal{L}_2^2$
        & $\mathcal{C}_2^2$
\end{tabular}
    \caption{Types of structures in $\mathcal{G}_2^2$ with two marks}
    \label{tab:typesG22}
\end{table}

\begin{remark}
There are no $2$-marked necklaces as they would require marking two arrows of the same type, which is forbidden. Also note that there are $2$-marked fixed points, but they are just a special case of chains with a single vertex. Similarly, we absorb $2$-marked mussels into loops, allowing the number of their vertices to start from $n=2$.
\end{remark}

\paragraph{Marked mussels} The EFG of a mussel marked by two marks is equal to $\tfrac12 m_1^2 \mu_2 t^2$. However, there are $4$ ways how we can mark one dashed and one undashed arrow so the total EFG is $2 m_1^2 \mu_2 t^2$.

\paragraph{Marked loops} The set of marked loops $\mathcal{L}_2^2$ with two marks is created by starting with a parallel pair of marked arrows and attaching a sequence of $n-1$ parallel pairs of arrows (where $n \geq 3$) to this marked pair of arrows. For each pair of arrows, one of the arrows is dashed and the other is undashed. We either have that all of the dashed and undashed arrows point in the same direction or we have that all of the dashed arrows point in one direction while all of the undashed arrows point in the opposite direction. Thus,
\begin{equation}
        \mathcal{L}^2_2 = \bigg{(}\oaloomdmup\!\!\!\!\!\! \star \textsc{Seq}_{\geq 1} \bigg{(}
        \oaloodup \!\!\!\!\!\!
        \bigg{)}\bigg{)}
    +
        \bigg{(}\oaloomdmuo\!\!\!\!\!\! \star \textsc{Seq}_{\geq 1} \bigg{(} \oalooduo
        \!\!\!\!\!\!
        \bigg{)}\bigg{)}.
    \end{equation}
In terms of the EGF,
\begin{equation*}
2 m_1^2 t \left(-1+\frac{1}{1-\mu_2 t}\right) = \frac{2 m_1^2 \mu_2 t^2}{1-\mu_2 t}.
\end{equation*}

\paragraph{Marked chains} Let $\mathcal{C}^2_2$ be the set of all marked chains. We can write
\begin{equation}
    \mathcal{C}^2_2 =
        \ofixmul \!\!     
\star  \textsc{Seq} \Bigg{(}
\obchadab\!\!\!\!\!\!+\!\!\!\obchadba\!\!\!\
\Bigg{)}
\star \fixmdar.
\end{equation}
In terms of the EGF,
\begin{equation*}
m_1\left(\frac{1}{1-2\mu_2^2 \frac{t^2}{2!}}\right) m_1 t =  \frac{m_1^2 t}{1-\mu_2^2 t^2}.
\end{equation*}

\paragraph{Pair of one-marked structures} Finally, we can create a multi-graph with two marks by combining a multi-graph with  no marks, a $1$-marked mussel with a mark on a dashed arrow, and a $1$-marked mussel with a mark on an undashed arrow. Since both of these mussels have EGF equal to $m_1\mu_3 t^2$, we get that the contribution to $G^2_2(t)$ is $(m_1\mu_3 t^2)^2 G_2(t)$.

\paragraph{Conclusion}
Adding the contributions together, we get
\begin{equation}
G_2^2(t) = \left(\frac{2 m_1^2 \mu_2 t^2}{1-\mu_2 t} + \frac{m_1^2 t}{1-\mu_2^2 t^2} + (m_1\mu_3 t^2)^2\right)G_2(t).
\end{equation}

\subsubsection{Full generating function}
\begin{proposition}
For any distribution of $X_{ij}$,
\begin{equation}
F_2(t) = \left(1+2 m_1 \mu _3 t^2+m_1^2 \left(\mu_3^2t^4+\frac{t}{1-\mu _2^2 t^2}+\frac{2 \mu _2 t^2}{1-\mu _2 t}\right)\right) \frac{\exp\left(\frac{(\mu _4-3\mu_2^2) t^2}{2}-\mu _2 t\right)}{\left(1-\mu _2 t\right){}^2 \sqrt{1-\mu _2^2 t^2}}.
\end{equation}
\end{proposition}
\begin{proof}
By adding up the contributions,
\begin{equation}
\begin{split}
F_2(t) & = G_2^0(t)+G_2^1(t)+G_2^2(t) \\
& = \left[1 + 2m_1\mu_3t^2 + \left(\frac{2 m_1^2 \mu_2 t^2}{1-\mu_2 t} + \frac{m_1^2 t}{1-\mu_2^2 t^2} + (m_1\mu_3 t^2)^2\right)\right]G_2(t).
\end{split}
\end{equation}
\end{proof}
\begin{corollary}
For any distribution $X_{ij}$ with $\mu_2 = 1$,
\begin{equation}
F_2(t)= \left(1+2 m_1 \mu _3 t^2+m_1^2 \left(\mu _3^2 t^4+\frac{t}{1-t^2}+\frac{2t^2}{1-t}\right)\right) \frac{\exp\left(\frac{(\mu _4-3) t^2}{2}-t\right)}{(1-t)^2 \sqrt{1-t^2}}.
\end{equation}
\end{corollary}

\subsubsection{Asymptotics}
Expanding the EGFs around their singular point and by singular asymptotics (c.f. \cite{flajolet2009analytic}),
\begin{equation}
f_2(n) = \frac{4 n^{3/2} n!}{3 \sqrt{2 \pi }} e^{\frac{1}{2} \left(\mu _4-5\right)} \left(m_1^2 n+1+2 m_1 \mu _3+\frac{m_1^2}{4} \left(43+4 \mu _3^2-10 \mu_4\right)+ O\left(\frac{1}{n}\right)\right).
\end{equation}

\section{Wigner matrices}\label{sec:Wigner}
In this section, we analyze the EGFs $H_1(t,a_1,c_1)$ and $H_2(t, a_1, a_2, c_1, c_2)$ of the characteristic correlation functions $h_1(n,a_1,c_1)$ and $h_2(n, a_1, a_2, c_1, c_2)$ for Wigner matrices $Z$. For this analysis, we use the same techniques as before with the following modification. Due to the ${c_1}Id$ and ${c_2}Id$ terms in $h_2(n, a_1, a_2, c_1, c_2) = \Exx \left[\det\left(Z + {a_1}J + {c_1}Id
\right)\det\left(Z  + {a_2}J + {c_2}Id
\right)
\right]$, tables with exactly one $i$ in column $i$ may now have nonzero weight. Thus, we modify the definition of trivial marked tables and multi-graphs as follows.
\begin{definition}
Given a marked $k \times n$ table $t$ and the corresponding marked multi-graph $\hat{t}$, we say that $t$ and $\hat{t}$ are trivial if there exist $i,j \in [n]$ such that there is exactly one unmarked edge (in either direction) between $i$ and $j$ in $\hat{t}$. Equivalently, $t$ and $\hat{t}$ are  trivial if there exist $i,j \in [n]$ such that the number of times $j$ appears in column $i$ of $t$ plus the number of times $i$ appears in column $j$ of $t$ is exactly $1$.

Otherwise, we say that $t$ and $\hat{t}$ are non-trivial.  
\end{definition}
In order to distinguish this case from our previous analysis where we had that $c_i = 0$ for all $i \in [k]$, we denote the sets of non-trivial marked tables and multi-graphs with $k$ rows/colors by $H_k^\times$ and $\mathcal{H}^\times_k$ rather than $G_k^\times$ and $\mathcal{G}^\times_k$. We have similar definitions and equations as in Section \ref{sec:markedmultigraphs}:
\begin{enumerate}
\item We define $\mathcal{H}^r_k$ to be the subset of $\mathcal{H}_k^\times$ of marked graphs which have exactly $r$ marks (when $r=0$, we sometimes omit writing $r$).
\item We define $h^r_k(n) = \sum_{\hat{t} \in \mathcal{H}_{k,n}^r} w(\hat{t}) \sgn(\hat{t})$ and its corresponding EGF $H_k^r(t) = \sum_{n=0}^\infty \frac{t^n}{n!} h^r_k(n)$.
\item We have that $H_k(t,a_1,\ldots,a_k,c_1,\ldots,c_k) = \sum_{r=0}^k H_k^r(t)$
\item As before, for any $\hat{g} \in \mathcal{H}_k$, we define the weight $W(\hat{g})$ to be $W(\hat{g}) = \sgn(\hat{g})w(\hat{g}){t^n}$ and we have that $H_k(t,a_1,\ldots,a_k,c_1,\ldots,c_k) = \sum_{\hat{g}\in \mathcal{H}^\times_k} \frac{W(\hat{g})}{|\hat{g}|!}$
\end{enumerate}
\begin{example}\label{ex:H2table}
An example of a nontrivial table $t \in H^{\times}_{2,9}$ and its corresponding nontrivial multi-graph $\hat{t} \in \mathcal{H}^{\times}_{2,9}$.
\begin{figure}[H]
\centering
\begin{minipage}{.48\textwidth}
        \centering
    \begin{tabular}{|c|c|c|c|c|c|c|c|c|}
    \hline
        1 & 2 & 3 & 4 & 5 & 6 & 7 & 8 & 9\\
    \hdashline
        3 & $\times$ & 1 & 9 & 8 & 7 & 6 & 5 & 4 \\
        3 & 6 & $\times$ & 5 & 4 & 2 & 7 & 9 & 8 \\
    \hline
    \end{tabular}
\caption{$t \in  H^2_{2,9}$ with weight $w(t) = c_2(\mu_2+\nu_2)^6(\mu_3 + \mathbf{i}\nu_3)a_1a_2 $,}
    \end{minipage}%
\hfill
\begin{minipage}{.48\textwidth}
        \centering
\begin{tikzpicture}[scale = 0.85]
\node[vertex] (1) at (-3.5,1) {$1$};
\node[vertex] (3) at (-3.5,-1) {$3$};
\node[vertex] (2) at (-1,-1) {$2$};
\node[vertex] (6) at (0,0) {$6$}; 
\node[vertex] (7) at (1,1) {$7$};
\node[vertex] (4) at (2.5,0) {$4$};
\node[vertex] (9) at (3.5,1) {$9$};
\node[vertex] (8) at (4.5,0) {$8$};
\node[vertex] (5) at (3.5,-1) {$5$};
\draw[->,dashed] (1) to [bend left=50] (3);
\draw[->] (1.-65) to [bend left=50] (3.65);
\draw[->,dashed] (3) to [bend left=50] node[fill=white,sloped]{$\!\!\!\times\!\!\!$} (1);
\draw[->] (3.115) to [bend left=50] (1.-115);
\draw[->,dashed] (9) to [bend left=20] (8);
\draw[->,dashed] (8) to [bend left=20] (9);
\draw[->] (8) to [bend left=20] (5);
\draw[->] (5) to [bend left=20] (8);
\draw[->,dashed] (5) to [bend left=20] (4);
\draw[->,dashed] (4) to [bend left=20] (5);
\draw[->] (4) to [bend left=20] (9);
\draw[->] (9) to [bend left=20] (4);
        \draw[->] (2.-70+180) arc (-130+180:125+180:0.45) node[fill=white,midway,sloped]{$\!\!\times\!\!$};
        \draw[->,dashed] (2) to [bend left=20] (6);
        \draw[->,dashed] (6) to [bend left=20] (2);
        \draw[->] (6) to [bend left=20] (7);
        \draw[->] (7) to [bend left=20] (6);
        \draw[->,dashed] (7.-70) arc (-130:125:0.45);
\end{tikzpicture}
\caption{The corresponding $\hat{t} \in \mathcal{H}^2_{2,9}$}
\label{fig:H2table}
    \end{minipage}%
\end{figure}
\end{example}
\subsection{First moment}
\subsubsection{Zero marks}
\begin{proposition}
\begin{equation}
H_1(t,0) = H_1^0(t) = e^{c_1 t-\frac12 (\mu_2+\nu_2) t^2}.
\end{equation}
\end{proposition}
\begin{proof}
We have the following structures in $\mathcal{H}_1^0$:
\begin{itemize}
    \item Fixed points ($1$-cycles) with weight $W = c_1 t$.
    \item Mussels ($2$-cycles) with weight $W = (\mu_2 + \nu_2)t^2$.
\end{itemize}
Thus,
\begin{equation}
        \mathcal{H}_1^0 = \textsc{Set}\bigg{(}\,
\begin{tikzpicture}[baseline=-0.4ex, scale = 0.8]
\node[vertex] (1) at (0,0) {$1$};
        \draw[->] (1.-70+180) arc (-130+180:125+180:0.45);
\end{tikzpicture}\!\!
\bigg{)}
\star
\textsc{Set}\Bigg{(}
\begin{tikzpicture}[baseline=-0.65ex, scale = 0.8]
\node[vertex] (2) at (0,0.8) {$2$};
\node[vertex] (1) at (0,-0.8) {$1$};
\draw[->] (2) to [bend left=50] (1);
\draw[->] (1) to [bend left=50] (2);
\end{tikzpicture}
\Bigg{)}.
\end{equation}
In terms of the EGF,
\begin{equation}
H_1^0(t) = H^0_1(t) = e^{c_1 t-\frac12 (\mu_2+\nu_2) t^2}.
\end{equation}
\end{proof}

\subsubsection{One mark}
The only nontrivial marked structure in $\mathcal{H}_1^1$ is a marked fixed point ($1$-cycle with one mark) with weight $W = a_1 t$, thus
\begin{equation}
        \mathcal{H}^1_1 = \fixmual \star\mathcal{H}_1^0.
\end{equation}
In terms of the EGF,
\begin{equation}
H^1_1(t) = a_1 t\,H^0_2(t) = a_1 t \, e^{c_1 t-\frac12 (\mu_2+\nu_2) t^2}.
\end{equation}
\subsubsection{Full generating function}
\begin{proposition}
\begin{equation}
H_1(t,a_1) = (1+a_1 t) e^{c_1 t-\frac12 (\mu_2+\nu_2) t^2}.
\end{equation}
\end{proposition}
\begin{proof}
Observe that $H_1(t,a_1) = H_1^0(t) + H_1^1(t) = (1+a_1 t) e^{c_1 t-\frac12 (\mu_2+\nu_2) t^2}$.    
\end{proof}

\subsection{Second moment}
\subsubsection{Zero marks}
\begin{proposition}
\begin{equation}
H_2(t,0,0) = H_2^0(t) = \frac{\exp \left(t \left(\kappa _2-2 \mu _2\right)+\frac{1}{2} t^2 \left(\mu _4-3 \nu _2^2-3\mu _2^2+\nu _4\right)-\frac{\left(\mu_2+\nu_2\right)\kappa_1^2t^2}{1+t \left(\mu _2+\nu _2\right)}\right)}{\left(1-2\mu_2t+(\mu _2^2-\nu _2^2)t^2\right)\sqrt{1-(\mu _2+\nu _2)^2 t^2}}.
\end{equation}
\end{proposition}
\begin{proof}
For $\mathcal{H}_2^0$, we have the following types of structures:
\begin{itemize}
    \item Fixed points (two $1$-cycles on the same element) with weight $W = (c_1c_2+\kappa_2) t$
    \item Mussels (two $2$-cycles on two elements) with weight $W = (\mu_4+2\mu_2\nu_2+\nu _4) t^2$
    \item Loops (double cycles on $n\geq 3$ vertices) with weight $W = (\mu_2+\nu_2)^n t^n$ if the arrows have the opposite direction and $W = (\mu_2-\nu_2)^n t^n$ if they have the same direction
    \item Necklaces (closed chains of alternating dashed/undashed $2$-cycles on $n$ vertices, $n \geq 4$ even) with weight $W = (\mu_2+\nu_2)^n t^n$.
    \item Chains (sequences of alternating dashed/undashed $2$-cycles on $n$ vertices with endpoints being $1$-cycles, $n\geq 3$ odd) with weight $W = (\mu_2+\nu_2)^{n-1} c_1c_2 t^n$
    \item Dumbbells (sequences of alternating dashed/undashed $2$-cycles on $n$ vertices with endpoints being $1$-cycles of same type, $n\geq 2$ even) with weight $W = -(\mu_2+\nu_2)^{n-1} c_1^2 t^n$ if those $1$-cycles are undashed. If they are dashed, we simply replace $c_1^2$ by $c_2^2$.
\end{itemize}
Representatives of each type are shown in Table \ref{tab:typesH20}.
\begin{table}[H]
    \centering
\begin{tabular}{cccccc}
        \begin{tikzpicture}[baseline = -1.3ex,scale = 0.85]
        \node[vertex] (1) at (-2.5,1) {$1$};
        \draw[->] (1.-60) arc (180+45+20:540-45-15:0.45);
        \draw[->,dashed] (1.-90) arc (180+45:540-45:0.6);
        \end{tikzpicture}
     &
        \begin{tikzpicture}[scale = 0.85]
        \node[vertex] (1) at (-2.5,1) {$1$};
        \node[vertex] (2) at (-2.5,-1) {$2$};
        \draw[->] (1.-65) to [bend left=50] (2.65);
        \draw[->,dashed] (2) to [bend left=50] (1);
        \draw[->] (2.115) to [bend left=50] (1.-115);
        \draw[->,dashed] (1) to [bend left=50] (2);
        \end{tikzpicture}
     &
        \begin{tikzpicture}[scale = 0.85]
        \node[vertex] (1) at (-1,0.8) {$1$};
        \node[vertex] (2) at (1,0.8) {$2$}; 
        \node[vertex] (3) at (0,-1) {$3$};
        \draw[->] (2.182) to [bend left=-40] (1.-002);
        \draw[->,dashed] (1) to [bend left=35] (2);
        \draw[->,dashed] (2) to [bend left=40] (3);
        \draw[->] (3.55) to [bend left=-40] (2.250);
        \draw[->,dashed] (3) to [bend left=40] (1);
        \draw[->] (1.290) to [bend left=-40] (3.130);
        \end{tikzpicture}      
     &
        \begin{tikzpicture}[scale = 0.85]
        \node[vertex] (4) at (3,1) {$4$};
        \node[vertex] (3) at (4,0) {$3$};
        \node[vertex] (2) at (3,-1) {$2$};
        \node[vertex] (1) at (2,0) {$1$};
        \draw[->,dashed] (4) to [bend left=20] (3);
        \draw[->,dashed] (3) to [bend left=20] (4);
        \draw[->] (3) to [bend left=20] (2);
        \draw[->] (2) to [bend left=20] (3);
        \draw[->,dashed] (2) to [bend left=20] (1);
        \draw[->,dashed] (1) to [bend left=20] (2);
        \draw[->] (1) to [bend left=20] (4);
        \draw[->] (4) to [bend left=20] (1);
        \end{tikzpicture}
    &
 \begin{tikzpicture}[scale = 0.85]
        \node[vertex] (1) at (0,0) {$1$};
        \node[vertex] (2) at (1.5,1) {$2$};
        \node[vertex] (3) at (3,2) {$3$};
        \draw[->] (1.-70+180) arc (-130+180:125+180:0.45);
        \draw[->,dashed] (1) to [bend left=20] (2);
        \draw[->,dashed] (2) to [bend left=20] (1);
        \draw[->] (2) to [bend left=20] (3);
        \draw[->] (3) to [bend left=20] (2);
        \draw[->,dashed] (3.-70) arc (-130:125:0.45);
        \end{tikzpicture}
        &
 \begin{tikzpicture}[scale = 0.85]
        \node[vertex] (1) at (0,0) {$1$};
        \node[vertex] (2) at (0,1.5) {$2$};
        \draw[->] (1.200) arc (-210:35:0.45);
        \draw[->,dashed] (1) to [bend left=30] (2);
        \draw[->,dashed] (2) to [bend left=30] (1);
        \draw[->] (2.10) arc (-30:215:0.45);
        \end{tikzpicture}
        \\
     Fixed point & Mussel & Loop & Necklace & Chain & Dumbbell
        \\
        $\mathcal{X}_2$
        & $\mathcal{M}_2$
        & $\mathcal{L}_2$
        & $\mathcal{N}_2$
        & $\mathcal{C}_2$
        & $\mathcal{D}_2$
\end{tabular}
    \caption{Types of structures in $\mathcal{H}_2$}
    \label{tab:typesH20}
\end{table}

\paragraph{Fixed points} Let $\mathcal{X}_2$ be the set of fixed points in $\mathcal{H}_2^0$. Then
\begin{equation}
    \mathcal{X}_2 = \textsc{Set} \bigg{(}\!\!\!
        \begin{tikzpicture}[baseline=3.7ex,scale = 0.85]
        \node[vertex] (1) at (-2.5,1) {$1$};
        \draw[->] (1.-60) arc (180+45+20:540-45-15:0.45);
        \draw[->,dashed] (1.-90) arc (180+45:540-45:0.6);
        \end{tikzpicture}
\bigg{)}.
\end{equation}
In terms of the EGF, this gives the factor of $\exp((c_1c_2+\kappa_2) t)$.

\paragraph{Mussels}
Next, let $\mathcal{M}_2$ be the set of mussels in $\mathcal{H}_2^0$. Then
\begin{equation}
    \mathcal{M}_2 = \textsc{Set}\Bigg{(}
        \begin{tikzpicture}[baseline=-0.65ex,scale = 0.85]
        \node[vertex] (1) at (-2.5,1) {$1$};
        \node[vertex] (2) at (-2.5,-1) {$2$};
        \draw[->] (1.-65) to [bend left=50] (2.65);
        \draw[->,dashed] (2) to [bend left=50] (1);
        \draw[->] (2.115) to [bend left=50] (1.-115);
        \draw[->,dashed] (1) to [bend left=50] (2);
        \end{tikzpicture}
\Bigg{)}.
\end{equation}
The EGF of the inner factor is $(\mu_4+2\mu_2\nu_2+\nu _4) t^2/2!$ so the EGF of $\mathcal{M}_2$ is $\exp(\tfrac12 (\mu_4+2\mu_2\nu_2+\nu _4) t^2)$.

\paragraph{Loops} Let $\mathcal{L}_2$ be the set of loops in $\mathcal{H}_2^0$. Note that $\mathcal{L}_2$ includes both possible orientations for the arrows. We have that 
\begin{equation}
    \mathcal{L}_2 = \textsc{Set} \bigg{(} \textsc{Cyc}_{\geq 3} \bigg{(} \oaloodup \!\!\!\!\!\!
\bigg{)}\bigg{)} \star
\textsc{Set} \bigg{(} \textsc{Cyc}_{\geq 3} \bigg{(} \oalooduo \!\!\!\!\!\!
\bigg{)}\bigg{)}.
\end{equation}
The corresponding EGF of $\mathcal{L}_2$ is equal to
\begin{equation*}
\textstyle\exp\left(-(\mu_2\!-\!\nu_2)t \!-\! \frac{(\mu_2\!-\!\nu_2)^2 t^2}{2} \!+\! \ln\left(\frac{1}{1\!-\!(\mu_2\!-\!\nu_2) t}\right)\right)\exp\left(-(\mu_2\!+\!\nu_2)t \!-\! \frac{(\mu_2\!+\!\nu_2)^2 t^2}{2} \!+\! \ln\left(\frac{1}{1\!-\!(\mu_2\!+\!\nu_2) t}\right)\right)
\end{equation*}
or after simplification,
\begin{equation*}
\frac{e^{-2\mu_2t-\left(\mu_2^2+\nu_2^2\right)t^2}}{1-2\mu_2 t+\left(\mu _2^2-\nu _2^2\right) t^2}.
\end{equation*}

\paragraph{Necklaces} For the set of necklaces $\mathcal{N}_2$,
\begin{equation}
    \mathcal{N}_2 = \textsc{Set} \Bigg{(} \frac12 \,  \textsc{Cyc}_{\geq 2} \Bigg{(}
    \obchauab\!\!\!\!\!\!
    + \!\!\! \obchauba \!\!\!\
\Bigg{)}\Bigg{)}.
\end{equation}
Hence, its EGF is
\begin{equation}
\exp\left(\frac12\left(-2(\mu_2+\nu_2)^2\frac{t^2}{2!}+\ln\left(\frac{1}{1-2(\mu_2+\nu_2)^2\frac{t^2}{2!}}\right)\right)\right) = \frac{e^{-\frac12 (\mu_2+\nu_2)^2 t^2}}{\sqrt{1-(\mu_2+\nu_2)^2 t^2}}.   
\end{equation}

\paragraph{Chains} Let $\mathcal{C}_2$ denote the set of all chains. We can write
\begin{equation}
    \mathcal{C}_2 =
        \textsc{Set} \Bigg{(} \ofixl \!\!     
\star  \textsc{Seq}_{\geq 1} \Bigg{(}
\obchadab\!\!\!\!\!\! + \!\!\!\obchadba\!\!\! \Bigg{)} \star \fixdar \Bigg{)}.
\end{equation}
In terms of the EGF,
\begin{equation*}
\exp\left[c_1\left(-1+\frac{1}{1-2(\mu_2+\nu_2)^2 \frac{t^2}{2!}}\right) c_2 t \right]= \exp \left(\frac{c_1 c_2 \left(\mu _2+\nu _2\right){}^2 t^3}{1-\left(\mu _2+\nu _2\right){}^2 t^2}\right).
\end{equation*}

\paragraph{Dumbbells} There are two types of dumbbells as the endpoints are either undashed $1$-cycles or dashed $1$-cycles. Dumbbells are different from chains as they are symmetric and have the opposite sign and parity of the number of vertices. Hence, we have that for the set $\mathcal{D}_2$ of all dumbbells,
\begin{equation}
\begin{split}    
    \mathcal{D}_2 = \,\, &
        \textsc{Set} \Bigg{(}\,\frac12\,\,\,
        \oachaua \,\,     
        \star  \textsc{Seq} \Bigg{(}
        \obchauab \!\!\!\!\!\! +
        \!\!\!\obchauba\!\!\!\
        \Bigg{)}
        \star\,\, \fixuap \, \Bigg{)}
    \\
        \star \,\, & \textsc{Set} \Bigg{(}\,\frac12\,\,\, \oachada \,\,     
        \star  \textsc{Seq} \Bigg{(}
        \obchadab \!\!\!\!\!\!
        + \!\!\! \obchadba \!\!\!\
        \Bigg{)}
        \star\,\,
        \fixdap \, \Bigg{)}.
    \end{split}
    \end{equation}
Its EGF is thus equal to
\begin{equation}
\begin{split}
&\left[\exp\left(-\frac12 c_1 (\mu_2+\nu_2)t\left(\frac{1}{1-2(\mu_2+\nu_2)^2 \frac{t^2}{2!}}\right) c_1 t\right)\right]\\
&\left[\exp\left(-\frac12 c_2 (\mu_2+\nu_2)t\left(\frac{1}{1-2(\mu_2+\nu_2)^2 \frac{t^2}{2!}}\right) c_2 t\right)\right]\\
& = \exp\left(-\,\frac{\frac12 (c_1^2+c_2^2)(\mu_2+\nu_2) t^2}{1-(\mu_2+\nu_2)^2t^2}\right).
\end{split}
\end{equation}

\paragraph{Conclusion}
As $\mathcal{H}_2^0 = \mathcal{X}_2 \star \mathcal{M}_2 \star \mathcal{L}_2 \star \mathcal{N}_2 \star \mathcal{C}_2 \star \mathcal{D}_2$, by multiplying the EGFs we obtain that
\begin{equation*}
H_2^0(t)= \frac{\exp \left(t \left(c_1 c_2\!+\!\kappa_2\!-\!2\mu_2\right)+\frac{1}{2} t^2 \left(\mu_4\!-\!3 \nu_2^2\!-\!3 \mu_2^2\!+\!\nu_4\right)-\frac{t^2}{2}\frac{(c_1^2+c_2^2) \left(\mu_2+\nu_2\right)}{1-t^2 \left(\mu _2+\nu _2\right){}^2}+\frac{c_1 c_2
   \left(\mu _2+\nu _2\right){}^2 t^3}{1-\left(\mu _2+\nu _2\right){}^2 t^2}\right)}{\left(1-2 \mu _2 t+\left(\mu _2^2-\nu _2^2\right)
   t^2\right) \sqrt{1-t^2 \left(\mu _2+\nu _2\right){}^2}}.
\end{equation*}
\end{proof}

\subsubsection{One mark}
A structure with one mark can be
\begin{itemize}
    \item A marked mussel (double $2$-loop with one arrow marked) where $W$ depends on both which type of arrow was dashed \emph{and} also if the marked arrow points from the smaller to the bigger number or the other way around. Overall, $W$ can attain one of the four values $a_1(\mu_3 \pm  \mathbf{i}\nu_3)t^2$ and $a_2(\mu_3 \pm \mathbf{i}\nu_3)t^2$
    \item A marked chain (alternating dashed/undashed $2$-cycle on an odd $n\geq 1$ number of vertices, provided we attach one marked $1$-cycle and one unmarked $1$-cycle to the endpoints) with weight $W = a_1 \kappa_1(\mu_2+\nu_2)^{n-1} t^n$ or $W = a_2 \kappa_1(\mu_2+\nu_2)^{n-1} t^n$ depending on which type of arrow is marked.
    \item A marked dumbbell (sequence alternating dashed/undashed $2$-cycles on $n$ vertices with endpoints being $1$-cycles of same type, $n\geq 2$ even, with one mark at the endpoint) with weight $W = -a_1 \kappa_1 (\mu_2+\nu_2)^{n-1} t^n$ or $W = -a_2 \kappa_1 (\mu_2+\nu_2)^{n-1} t^n$.
\end{itemize}
\begin{table}[H]
    \centering
\begin{tabular}{ccc}
        \begin{tikzpicture}[scale = 0.85]
        \node[vertex] (1) at (-2.5,1) {$1$};
        \node[vertex] (2) at (-2.5,-1) {$2$};
        \draw[->] (1.-65) to [bend left=50] node[fill=white,sloped]{$\!\!\!\times\!\!\!$} (2.65);
        \draw[->,dashed] (2) to [bend left=50] (1);
        \draw[->] (2.115) to [bend left=50] (1.-115);
        \draw[->,dashed] (1) to [bend left=50] (2);
        \end{tikzpicture}
    &
        \begin{tikzpicture}[scale = 0.85]
        \node[vertex] (1) at (0,0) {$1$};
        \node[vertex] (2) at (1.3,1) {$2$};
        \node[vertex] (3) at (2.6,2) {$3$};
        \draw[->] (1.-70+180) arc (-130+180:125+180:0.45) node[fill=white,midway,sloped]{$\!\!\times\!\!$};
        \draw[->,dashed] (1) to [bend left=20] (2);
        \draw[->,dashed] (2) to [bend left=20] (1);
        \draw[->] (2) to [bend left=20] (3);
        \draw[->] (3) to [bend left=20] (2);
        \draw[->,dashed] (3.-70) arc (-130:125:0.45);
        \end{tikzpicture}
    &
        \begin{tikzpicture}[scale = 0.85]
        \node[vertex] (1) at (0,0) {$1$};
        \node[vertex] (2) at (0,1.5) {$2$};
        \draw[->] (1.200) arc (-210:35:0.45) node[fill=white,midway,sloped]{$\!\!\times\!\!$};
        \draw[->,dashed] (1) to [bend left=30] (2);
        \draw[->,dashed] (2) to [bend left=30] (1);
        \draw[->] (2.10) arc (-30:215:0.45);
        \end{tikzpicture}
    \\
     $1$-Marked mussel & $1$-Marked chain & $1$-Marked dumbbell
        \\
        $\mathcal{M}_2^1$
        & $\mathcal{C}_2^1$
        & $\mathcal{D}_2^1$
\end{tabular}
    \caption{Types of structures in $\mathcal{H}_2^1$ with one mark}
    \label{tab:typesH21}
\end{table}
\begin{remark}
Here, we included $1$-marked fixed points (one marked and one unmarked $1$-cycle) into $1$-marked chains as they have the same weight.
\end{remark}

\paragraph{Marked mussels}
Next, let $\mathcal{M}_2^1$ be the set of $1$-marked mussels. By adding up the contribution for the four different possibilities for which edge is marked, that is
\begin{equation}
\mathcal{M}_2^1 = \musmuab\, + \, \musmuba\, +\, \musmdab + \musmdba,
\end{equation}
we obtain $2 (a_1 + a_2) \mu_3 t^2/2!$ for its EGF.

\paragraph{Marked chains} For the set $\mathcal{C}_2^1$ of chains with one mark, we write
\begin{equation}
\begin{split}
        \mathcal{C}_2^1 = \, & \ofixmul\!\! \star \textsc{Seq} \Bigg{(} \obchadab\!\!\!\!\!\! + \!\!\!\obchadba\!\!\!\ \Bigg{)} \star \fixdar.
    \\
        + \, & \ofixmdl \!\! \star  \textsc{Seq} \Bigg{(} \obchauab\!\!\!\!\!\! + \!\!\!\obchauba\!\!\!\ \Bigg{)} \star \fixuar.
\end{split}
    \end{equation}
In terms of the EGF,
\begin{equation*}
a_1\left(\frac{1}{1-2(\mu_2+\nu_2)^2 \frac{t^2}{2!}}\right) c_2 t + a_2\left(\frac{1}{1-2(\mu_2+\nu_2)^2 \frac{t^2}{2!}}\right) c_1 t=  \frac{(a_1 c_2+a_2 c_1) t}{1-(\mu_2+\nu_2)^2 t^2}.
\end{equation*}

\paragraph{Marked dumbbells} For the set $\mathcal{D}_2^1$ of all dumbbells with one mark,
\begin{equation}
\begin{split}
        \mathcal{D}_2^1 = \,\, & \oachamua \,\,\, \star \, \textsc{Seq} \Bigg{(} \obchauab\!\!\!\!\!\! + \!\!\!\obchauba\!\!\! \Bigg{)} \star\,\, \fixuap \\
        + \,\, & \oachamda\,\,\, \star \, \textsc{Seq} \Bigg{(} \obchadab\!\!\!\!\!\! + \!\!\!\obchadba\!\!\! \Bigg{)} \star\,\, \fixdap.
    \end{split}
    \end{equation}
Its EGF is thus equal to
\begin{equation}
\begin{split}
    & -a_1 (\mu_2+\nu_2)t\left(\frac{1}{1-2(\mu_2+\nu_2)^2 \frac{t^2}{2!}}\right) c_1 t -a_2 (\mu_2+\nu_2)t\left(\frac{1}{1-2(\mu_2+\nu_2)^2 \frac{t^2}{2!}}\right) c_2 t \\
    & = -\,\frac{(\mu_2+\nu_2)(a_1c_1+a_2c_2) t^2}{1-(\mu_2+\nu_2)^2t^2}.
\end{split}
\end{equation}
Note that the sign is negative as the total number of $2$-cycles contained in a dumbbell (of either color) is odd.
\paragraph{Conclusion}
Since $\mathcal{H}_2^1 = (\mathcal{M}^1_2 + \mathcal{C}^1_2 + \mathcal{D}^1_2) \star \mathcal{H}_2^0$, we obtain that
\begin{equation}
H_2^1(t)= \left((a_1 + a_2) \mu_3 t^2+\frac{(a_1c_2+a_2c_1) t}{1-(\mu_2+\nu_2)^2 t^2} -\frac{(\mu_2+\nu_2)(a_1c_1+a_2c_2) t^2}{1-(\mu_2+\nu_2)^2t^2}\right) H_2^0(t).
\end{equation}

\subsubsection{Two marks}
The following Table \ref{tab:typesH22} shows representatives of all types of the remaining (connected) structures having exactly two marks. A structure with two marks can be
\begin{itemize}
    \item A marked loop (double cycle on $n\geq 2$ vertices with two marks) with weight $W = a_1 a_2(\mu_2 \pm \nu_2)^{n-1} t^n$ (plus sign if the arrows have the opposite direction and minus sign when they have the same direction)
    \item A marked chain (alternating dashed/undashed $2$-cycle on an odd number $n \geq 1$ of vertices, provided we attach two marked $1$-cycles to the endpoints) with weight $W = a_1 a_2(\mu_2+\nu_2)^{n-1} t^n$.
\end{itemize}
\begin{table}[H]
    \centering
\begin{tabular}{cc}
        \begin{tikzpicture}[scale = 0.85]
        \node[vertex] (1) at (-1,0.8) {$1$};
        \node[vertex] (2) at (1,0.8) {$2$}; 
        \node[vertex] (3) at (0,-1) {$3$};
        \draw[->] (2.182) to [bend left=-40] (1.-002);
        \draw[->,dashed] (1) to [bend left=35] node[fill=white,sloped]{$\!\!\!\times\!\!\!$} (2);
        \node[vertex] (x) at (0,1) {$\times$};
        \draw[->,dashed] (2) to [bend left=40] (3);
        \draw[->] (3.55) to [bend left=-40] (2.250);
        \draw[->,dashed] (3) to [bend left=40] (1);
        \draw[->] (1.290) to [bend left=-40] (3.130);
        \end{tikzpicture}    
    &
        \begin{tikzpicture}[scale = 0.85]
        \node[vertex] (1) at (0,0) {$1$};
        \node[vertex] (2) at (1.5,1) {$2$};
        \node[vertex] (3) at (3,2) {$3$};
        \draw[->] (1.-70+180) arc (-130+180:125+180:0.45) node[fill=white,midway,sloped]{$\!\!\times\!\!$};
        \draw[->,dashed] (1) to [bend left=20] (2);
        \draw[->,dashed] (2) to [bend left=20] (1);
        \draw[->] (2) to [bend left=20] (3);
        \draw[->] (3) to [bend left=20] (2);
        \draw[->,dashed] (3.-70) arc (-130:125:0.45) node[fill=white,midway,sloped]{$\!\!\times\!\!$};
        \end{tikzpicture}
    \\
     $2$-Marked loop & $2$-Marked chain
        \\
        $\mathcal{L}_2^2$
        & $\mathcal{C}_2^2$
\end{tabular}
    \caption{Types of structures in $\mathcal{H}_2^2$ with two marks}
    \label{tab:typesH22}
\end{table}

\begin{remark}
There are no marked dumbbells as they would violate the assumption that there is at most one mark per one type of arrow. There are, however, $2$-marked fixed points (two marked $1$-cycle), but they are again counted as $2$-marked chains with one vertex. Similarly, $2$-marked mussels (two $2$-cycles with two marks) are counted as $2$-marked loops with two vertices.
\end{remark}

\paragraph{Marked loops} The set of marked loops $\mathcal{L}_2^2$ is created as follows
\begin{equation}
        \mathcal{L}^2_2 = \bigg{(}\oaloomdmup\!\!\!\!\!\! \star \textsc{Seq}_{\geq 1} \bigg{(}
        \oaloodup \!\!\!\!\!\!
        \bigg{)}\bigg{)}
    +
        \bigg{(}\oaloomdmuo\!\!\!\!\!\! \star \textsc{Seq}_{\geq 1} \bigg{(} \oalooduo
        \!\!\!\!\!\!
        \bigg{)}\bigg{)}.
    \end{equation}
In terms of the EGF,
\begin{equation*}
a_1a_2 t \left(-1+\frac{1}{1-(\mu_2-\nu_2) t}\right) + a_1a_2 t \left(-1+\frac{1}{1-(\mu_2+\nu_2) t}\right)
= \frac{2 a_1 a_2 t^2 \left(\mu_2-\mu_2^2 t+\nu_2^2 t\right)}{1-2\mu_2t+(\mu_2^2-\nu_2^2)t^2}.
\end{equation*}

\paragraph{Marked chains} Letting $\mathcal{C}^2_2$ be the set of all marked chains, we can write
\begin{equation}
    \mathcal{C}^2_2 =
        \ofixmul \!\!     
\star  \textsc{Seq} \Bigg{(}
\obchadab \!\!\!\!\!\! + \!\!\!\obchadba\!\!\!\
\Bigg{)}
\star \fixmdar.
    \end{equation}
In terms of the EGF,
\begin{equation*}
a_1\left(\frac{1}{1-2(\mu_2+\nu_2)^2 \frac{t^2}{2!}}\right) a_2 t =  \frac{a_1a_2 t}{1-(\mu_2+\nu_2)^2 t^2}.
\end{equation*}

\paragraph{Pair of one-marked structures} 
Finally, let $\mathcal{P}_2^2$ be the set of pairs of structures where one structure has an undashed arrow marked and the other structure has a dashed arrow marked. In other words,
\begin{equation}
\begin{split}
\mathcal{P}_2^2 = \,\,\, &
    \Bigg{[}\, \musmuab \,\, + \,\, \musmuba + \ofixmul\!\! \star \textsc{Seq} \Bigg{(} \obchadab\!\!\!\!\!\! + \!\!\!\obchadba\!\!\!\ \Bigg{)} \star \fixdar
    \\
    \,\, & + \oachamua \,\,\, \star \, \textsc{Seq} \Bigg{(} \obchauab\!\!\!\!\!\! + \!\!\!\obchauba\!\!\! \Bigg{)} \star\,\, \fixuap \Bigg{]}\\
    \star \, \, & \Bigg{[}\, \musmdab \,\, + \musmdba + \ofixmdl \!\! \star  \textsc{Seq} \Bigg{(} \obchauab\!\!\!\!\!\! + \!\!\!\obchauba\!\!\!\ \Bigg{)} \star \fixuar
    \\
    \,\, & + \oachamda\,\,\, \star \, \textsc{Seq} \Bigg{(} \obchadab\!\!\!\!\!\! + \!\!\!\obchadba\!\!\! \Bigg{)} \star\,\, \fixdap \Bigg{]}.
    \end{split}
\end{equation}
For its EGF, we have
\begin{equation*}
\begin{split}
&\left[a_1 \mu_3 t^2 +a_1\left(\frac{1}{1-2(\mu_2+\nu_2)^2 \frac{t^2}{2!}}\right) c_2 t -a_1 (\mu_2+\nu_2)t\left(\frac{1}{1-2(\mu_2+\nu_2)^2 \frac{t^2}{2!}}\right) c_1 t \right]\\
& \left[a_2 \mu_3 t^2 + a_2\left(\frac{1}{1-2(\mu_2+\nu_2)^2 \frac{t^2}{2!}}\right) c_1 t -a_2 (\mu_2+\nu_2)t\left(\frac{1}{1-2(\mu_2+\nu_2)^2 \frac{t^2}{2!}}\right) c_2 t\right],
\end{split}
\end{equation*}
which after simplification is equal to
\begin{equation*}
    a_1a_2\left(\mu_3 t^2\!+\!\frac{c_2 t}{1\!-\!\left(\mu_2\!+\!\nu_2\right){}^2 t^2}\!-\!\frac{c_1 \left(\mu _2+\nu _2\right) t^2}{1\!-\!\left(\mu _2\!+\!\nu
    _2\right){}^2 t^2}\right) \left(\mu_3 t^2\!+\!\frac{c_1 t}{1\!-\!\left(\mu_2\!+\!\nu_2\right){}^2 t^2}\!-\!\frac{c_2 \left(\mu _2+\nu _2\right)
    t^2}{1\!-\!\left(\mu _2\!+\!\nu _2\right){}^2 t^2}\right).
\end{equation*}

\paragraph{Conclusion}
Adding the contributions together, since $\mathcal{H}_2^2 = (\mathcal{L}_2^2+\mathcal{C}_2^2+\mathcal{P}_2^2) \star \mathcal{H}_2^0$, we obtain that
\begin{align*}
H_2^2(t) =  a_1 a_2 & \left[ \frac{2t^2 \left(\mu_2-\mu_2^2 t+\nu_2^2 t\right)}{1-2\mu_2t+(\mu_2^2-\nu_2^2)t^2}+\frac{t}{1-(\mu_2+\nu_2)^2 t^2}+\left(\mu_3 t^2 + \frac{c_2 t}{1\!-\!\left(\mu_2\!+\!\nu_2\right){}^2 t^2} \right.\right. \\
& \left.\left. \,\, - \,\frac{c_1 \left(\mu _2+\nu _2\right) t^2}{1\!-\!\left(\mu _2\!+\!\nu_2\right){}^2 t^2}\right) \left(\mu_3 t^2\!+\!\frac{c_1 t}{1\!-\!\left(\mu_2\!+\!\nu_2\right){}^2 t^2}\!-\!\frac{c_2 \left(\mu _2+\nu _2\right) t^2}{1\!-\!\left(\mu _2\!+\!\nu _2\right){}^2 t^2}\right)\right]\!H_2^0(t).
\end{align*}

\subsubsection{Full generating function}
\begin{corollary}
Adding up the contributions, we obtain that for any distribution $X_{ij}$, $Y_{ij}$ and $Z_{ii}$,
\begin{align*}
& H_2(t,a_1,a_2) = H_2^0(t)+H_2^1(t)+H_2^2(t) = \Bigg{[}1+(a_1 + a_2) \mu_3 t^2+\frac{(a_1c_2+a_2c_1) t}{1-(\mu_2+\nu_2)^2 t^2} \\
& - \left. \frac{(\mu_2+\nu_2)(a_1c_1+a_2c_2) t^2}{1-(\mu_2+\nu_2)^2t^2} + a_1 a_2 \Bigg{(}\frac{2t^2 \left(\mu_2-\mu_2^2 t+\nu_2^2 t\right)}{1-2\mu_2t+(\mu_2^2-\nu_2^2)t^2}+\frac{t}{1-(\mu_2+\nu_2)^2 t^2} + \bigg{(}\mu_3 t^2 \right.\\
& + \frac{c_2 t}{1\!-\!\left(\mu_2\!+\!\nu_2\right){}^2 t^2} - \frac{c_1 \left(\mu _2+\nu _2\right) t^2}{1\!-\!\left(\mu _2\!+\!\nu_2\right){}^2 t^2}\bigg{)} \bigg{(}\mu_3 t^2\!+\!\frac{c_1 t}{1\!-\!\left(\mu_2\!+\!\nu_2\right){}^2 t^2}\!-\!\frac{c_2 \left(\mu _2+\nu _2\right) t^2}{1\!-\!\left(\mu _2\!+\!\nu _2\right){}^2 t^2}\bigg{)}\Bigg{)}\Bigg{]} H_2^0(t).
\end{align*}
\end{corollary}

\section{Hermitian matrices}\label{sec:Hermitian}
\subsection{Eulerian numbers and polynomials}
As we will see in the following subsections, there are new degrees of freedom from suppressing the condition that the real and imaginary parts of the matrix $Z$ are independent. This increases the complexity of our generating functions. It is convenient to introduce here the so called \textit{Eulerian numbers} $A(n,k)$, \textit{Eulerian polynomials} $A_n(u)$, and their corresponding bivariate generating functions. Eulerian numbers can be defined via an explicit relation
\begin{equation}
A(n,k)=\sum^n_{i=0} (-1)^i \binom{n+1}{i}(k+1-i)^n
\end{equation}
They appear ubiquitously in various combinatorial problems. For our purposes, we need to know that $A(n,k)$ counts the number of ways to order the numbers  $[n] = \{1,2,\ldots,n\}$ so that there are exactly $k$ numbers which are greater than the previous number. As a consequence, $\sum_{k=0}^n A(n,k) = n!$, where by definition we take $A(n,n) = 0$ when $n \geq 1$. The Eulerian numbers may also be defined as the coefficients of the Eulerian polynomials
\begin{equation}
A_n(u)=\sum^n_{k=0}A(n,k)u^k,
\end{equation}
which are themselves coefficients of the corresponding bivariate exponential generating function 
\begin{equation*}
E(u,t) = \sum^{\infty}_{n=0}A_n(u)\frac{t^n}{n!}=\frac{u-1}{u-e^{(u-1)t}}.
\end{equation*}
In our analysis, we will use the following bivariate generating function. The reason this bivariate generating function is useful is that it allows us to handle loops of length $n \geq 3$ where the weight of the loop depends on whether each edge is ascending or descending.
\begin{definition}
We define $B(u,t) = u\sum_{n=3}^{\infty}{A_{n-1}(u)\frac{t^n}{n!}}$.
\end{definition}
\begin{proposition}\label{prop:Bmotivation}
Let $\mathcal{L}$ be the set of loops of length at least $3$ where the vertices of each loop are labeled by the indices $[n]$ where $n$ is the length of the loop. 
\begin{enumerate}
\item If the weight of a loop $L \in \mathcal{L}$ of order $n$ is 
$w_L = (-1)^{n-1}\lambda^{|\{e = (j,j') \in E(L): j' > j\}|}{\lambda'}^{|\{e = (j,j') \in E(L): j' < j\}|}$ (i.e., each edge $(j,j')$ such that $j' > j$ has a factor of $\lambda$, each edge $(j,j')$ such that $j' < j$ has a factor of $\lambda'$, and the sign for $L$ is $(-1)^{n-1}$)
then the exponential generating function for $\mathcal{L}$ is $-B\left(\frac{\lambda}{\lambda'},-{\lambda'}t\right)$.
\item If the weight of a loop $L \in \mathcal{L}$ of order $n$ is 
$w_L = \lambda^{|\{e = (j,j') \in E(L): j' > j\}|}{\lambda'}^{|\{e = (j,j') \in E(L): j' < j\}|}$ (i.e., the same as the previous case but the sign of each loop is now positive) then the exponential generating function for $\mathcal{L}$ is $B\left(\frac{\lambda}{\lambda'},{\lambda'}t\right)$.
\end{enumerate}
\end{proposition}
\begin{proof}
Observe that in the first case, the exponential generating function for $\mathcal{L}$ is 
\begin{equation}
\sum_{n=3}^{\infty}{\sum_{k=0}^{n-2}{A(n-1,k)\lambda^{k+1}{\lambda'}^{n-k-1}\frac{(-1)^{n-1}t^n}{n!}}} = -\frac{\lambda}{\lambda'}\sum_{n=3}^{\infty}{A_{n-1}\left(\frac{\lambda}{\lambda'}\right)\frac{(-\lambda't)^n}{n!}} = -B\left(\frac{\lambda}{\lambda'},-{\lambda'}t\right).
\end{equation}
To see this, observe that for each loop $L \in \mathcal{L}$ of length $n$, we can consider $n$ to be the starting and ending point of the loop. We are guaranteed to have a factor of ${\lambda}{\lambda'}$ from the two edges incident to $n$. The remaining edges are an ordering of $[n-1] = \{1,\ldots,n-1\}$ so for each $n \geq 3$ and $k \in [0,n-2]$, there are $A(n-1,k)$ loops of length $n$ which have $k+1$ edges $(j,j')$ where $j' > j$ and $n-k-1$ edges $(j,j')$ where $j' < j$.

Following similar logic, in the second case the exponential generating function for $\mathcal{L}$ is 
\begin{equation}
\sum_{n=3}^{\infty}{\sum_{k=0}^{n-2}{A(n-1,k)\lambda^{k+1}{\lambda'}^{n-k-1}\frac{t^n}{n!}}} = \frac{\lambda}{\lambda'}\sum_{n=3}^{\infty}{A_{n-1}\left(\frac{\lambda}{\lambda'}\right)\frac{(\lambda't)^n}{n!}} = B\left(\frac{\lambda}{\lambda'},{\lambda'}t\right).
\end{equation}
\end{proof}
We can evaluate the bivariate generating function $B(u,t)$ by integrating $E(u,t)$ with respect to $t$.
\begin{proposition}\label{prop:Bexpression}
$B(u,t) = u\sum^{\infty}_{n=3}A_{n-1}(u)\frac{t^n}{n!} = \ln\left(\frac{1-u}{1-ue^{(1-u)t}}\right) - ut - \frac{ut^2}{2}.$
\end{proposition}
\begin{proof}
\begin{equation}\label{Eq:EulerSpec}
\int{E(u,t)dt} = \sum^{\infty}_{n=1}A_{n-1}(u)\frac{t^n}{n!} = \frac{1}{u}\ln\left(\frac{1-u}{1-ue^{(1-u)t}}\right).
\end{equation}
This implies that $t + \frac{t^2}{2} + \frac{1}{u}B(u,t) = \frac{1}{u}\ln\left(\frac{1-u}{1-ue^{(1-u)t}}\right)$. Rearranging and multiplying both sides of this equation by $u$, we have that $B(u,t) = \ln\left(\frac{1-u}{1-ue^{(1-u)t}}\right) - ut - \frac{ut^2}{2}$, as needed.
\end{proof}
\begin{remark}
While it is not obvious from this expression for $B(u,t)$, it is not hard to check that $B(\frac{\lambda}{\lambda'},-{\lambda'}t) = B(\frac{\lambda'}{\lambda},-{\lambda}t)$
\end{remark}
\subsection{First Moment}
As a warm-up, we compute $\mathbb{E}[\det(Z)]$. For this computation, we do not assume that $\lambda_{10} = 0$ or $\kappa_1 = 0$ so we do not have $a_1$ or $c_1$ and we do not use marked tables or graphs. To indicate this, we write $\mathcal{F}_1$ and $F_1$ as we did in Section \ref{sec:multigraphs} rather than writing $\mathcal{H}_1$ and $H_1$.
\paragraph{Fixed points} Let $\mathcal{X}_1$ be the set of fixed points in $\mathcal{F}_1$. Then
\begin{equation}
    \mathcal{X}_1 = \textsc{Set} \bigg{(}\!\!\!
\begin{tikzpicture}[baseline=-0.4ex, scale = 0.8]
\node[vertex] (1) at (0,0) {$1$};
\draw[->] (1.-30) arc (-110:160:0.45);
\end{tikzpicture}
\bigg{)}.
\end{equation}
In terms of the EGF, this gives a factor of $\exp(\kappa_1 t)$.

\paragraph{Mussels} Next, let $\mathcal{M}_1$ be the set of mussels in $\mathcal{F}_1$. Then
\begin{equation}
    \mathcal{M}_1 = \textsc{Set}\Bigg{(}
\begin{tikzpicture}[baseline=-0.65ex, scale = 0.8]
\node[vertex] (2) at (0,0.8) {$2$};
\node[vertex] (1) at (0,-0.8) {$1$};
\draw[->] (2) to [bend left=50] (1);
\draw[->] (1) to [bend left=50] (2);
\end{tikzpicture}
\Bigg{)}.
\end{equation}
The EGF of the inner factor is $-\lambda_{11} t^2/2!$ so the EGF of $\mathcal{M}_1$ is $\exp(-\tfrac12 \lambda_{11} t^2)$.

\paragraph{Loops} For the set of loops $\mathcal{L}_1$, observe that $\mathcal{L}_1 = \textsc{Set}(\mathcal{L})$ where $\mathcal{L}$ is the set of single loops of length $n \geq 3$ and the weight of a loop $L \in \mathcal{L}$ of order $n$ is 
\[
w_L = (-1)^{n-1}\lambda^{|\{e = (j,j') \in E(L): j' > j\}|}{\lambda'}^{|\{e = (j,j') \in E(L): j' < j\}|}.
\]
where $\lambda = \lambda_{10}$ and $\lambda' = \lambda_{01}$. By Proposition \ref{prop:Bmotivation} and Proposition \ref{prop:Bexpression}, the EGF for $\mathcal{L}$ is 
\[
-B\left(\frac{\lambda_{10}}{\lambda_{01}},-\lambda_{01}t\right) = -\ln\left(\frac{\lambda_{01}-\lambda_{10}}{\lambda_{01}-\lambda_{10}e^{(\lambda_{10}-\lambda_{01})t}}\right) - \lambda_{10}t + \lambda_{10}\lambda_{01}\frac{t^2}{2}.
\]
Exponentiating this gives the EGF for $\mathcal{L}_1$, which is 
\[
\frac{\lambda_{01}-\lambda_{10}e^{(\lambda_{10}-\lambda_{01})t}}{\lambda_{01}-\lambda_{10}}\exp\left(-\lambda_{10}t + \lambda_{10}\lambda_{01}\frac{t^2}{2}\right).
\]

\paragraph{Conclusion}
Since $\mathcal{F}_1=\mathcal{X}_1 \star \mathcal{M}_1 \star \mathcal{L}_1$, by multiplying the EGFs we obtain that
\begin{align*}
F_1(t) = &\exp(\kappa_{1} t)\exp(-\tfrac12 \lambda_{11} t^2)\frac{\lambda_{01}-\lambda_{10}e^{(\lambda_{10}-\lambda_{01})t}}{\lambda_{01}-\lambda_{10}}\exp\left(-\lambda_{10}t + \lambda_{10}\lambda_{01}\frac{t^2}{2}\right)\\
&=\frac{e^{\frac{1}{2} t \left( 2 \kappa_1 + (-2 + t \lambda_{01}) \lambda_{10} - t \lambda_{11} \right)} \left( \lambda_{01} - e^{t (-\lambda_{01} + \lambda_{10})} \lambda_{10} \right)}{\lambda_{01} - \lambda_{10}}.
\end{align*}

\subsection{Second moment}
\subsubsection{Zero marks}
In the case of second moment, note that the Wigner case is very close to the more general Hermitian case (see Table \ref{tab:Wigner}). Thus, multi-graphs of Hermitian matrices have the same (connected) structures as the Wigner multi-graphs. This time, however, we no longer assume $\lambda_{20}=\lambda_{02}$ (both are equal to $\mu_2 - \nu_2$ in Wigner case). This extra degree of freedom affects how weight is calculated for loops whose arrows (dashed and undashed) point in the opposite direction (other structures have the same weight as before, we just replace $\mu$'s and $\nu$'s by corresponding $\lambda$'s). It is therefore convenient to split the loops $\mathcal{L}_2$ into $\mathcal{A}_2 \star \mathcal{O}_2$, where $\mathcal{A}_2$ denotes the set of anti-parallel loops (arrows pointing in different directions) and $\mathcal{O}_2$ denotes the set of parallel loops (with arrows poiting in the same direction). Table \ref{tab:typesH20Herm} below shows representatives of each (connected) structure.
\begin{table}[H]
    \centering
    \begin{tabular}{%(a,b) when overfloow occurs
      >{\centering}p{3.61em}%(3.6,3.61)
      >{\centering}p{3.42em}%(3.41,3.42)
      >{\centering}p{5.5em}%(5.48,5.5)
      >{\centering}p{7.02em}%(,7.02)
      >{\centering}p{5.52em}%(5.48,5.52)
      >{\centering}p{5.35em}%(5.30,5.35)
      >{\centering\arraybackslash}p{4.75em}%(,4.75)
      }
        \begin{tikzpicture}[baseline = -1.3ex,scale = 0.85]
        \node[vertex] (1) at (-2.5,1) {$1$};
        \draw[->] (1.-60) arc (180+45+20:540-45-15:0.45);
        \draw[->,dashed] (1.-90) arc (180+45:540-45:0.6);
        \end{tikzpicture}
     &
        \begin{tikzpicture}[scale = 0.85]
        \node[vertex] (1) at (-2.5,1) {$1$};
        \node[vertex] (2) at (-2.5,-1) {$2$};
        \draw[->] (1.-65) to [bend left=50] (2.65);
        \draw[->,dashed] (2) to [bend left=50] (1);
        \draw[->] (2.115) to [bend left=50] (1.-115);
        \draw[->,dashed] (1) to [bend left=50] (2);
        \end{tikzpicture}
     &
        \begin{tikzpicture}[scale = 0.85]
        \node[vertex] (1) at (-1,0.8) {$1$};
        \node[vertex] (2) at (1,0.8) {$2$}; 
        \node[vertex] (3) at (0,-1) {$3$};
        \draw[->] (1.-002) to [bend left=40] (2.182);
        \draw[->,dashed] (1) to [bend left=35] (2);
        \draw[->,dashed] (2) to [bend left=40] (3);
        \draw[->] (2.250) to [bend left=40] (3.55);
        \draw[->,dashed] (3) to [bend left=40] (1);
        \draw[->] (3.130) to [bend left=40] (1.290);
        \end{tikzpicture}      
     &
        \begin{tikzpicture}[scale = 0.85]
        \node[vertex] (1) at (-1,0.8) {$1$};
        \node[vertex] (2) at (1,0.8) {$2$}; 
        \node[vertex] (3) at (0,-1) {$3$};
        \draw[->] (2.182) to [bend left=-40] (1.-002);
        \draw[->,dashed] (1) to [bend left=35] (2);
        \draw[->,dashed] (2) to [bend left=40] (3);
        \draw[->] (3.55) to [bend left=-40] (2.250);
        \draw[->,dashed] (3) to [bend left=40] (1);
        \draw[->] (1.290) to [bend left=-40] (3.130);
        \end{tikzpicture}      
     &
        \begin{tikzpicture}[scale = 0.85]
        \node[vertex] (4) at (3,1) {$4$};
        \node[vertex] (3) at (4,0) {$3$};
        \node[vertex] (2) at (3,-1) {$2$};
        \node[vertex] (1) at (2,0) {$1$};
        \draw[->,dashed] (4) to [bend left=20] (3);
        \draw[->,dashed] (3) to [bend left=20] (4);
        \draw[->] (3) to [bend left=20] (2);
        \draw[->] (2) to [bend left=20] (3);
        \draw[->,dashed] (2) to [bend left=20] (1);
        \draw[->,dashed] (1) to [bend left=20] (2);
        \draw[->] (1) to [bend left=20] (4);
        \draw[->] (4) to [bend left=20] (1);
        \end{tikzpicture}
    &
 \begin{tikzpicture}[baseline=-6.0ex,scale = 0.85]
        \node[vertex] (1) at (-0.8,-1) {$1$};
        \node[vertex] (2) at (0,0) {$2$};
        \node[vertex] (3) at (0.8,1) {$3$};
        \draw[->,dashed] (1) to [bend left=20] (2);
        \draw[->,dashed] (2) to [bend left=20] (1);
        \draw[->] (2) to [bend left=20] (3);
        \draw[->] (3) to [bend left=20] (2);
        \draw[->] (1.200) arc (-210:35:0.45);
        \draw[->,dashed] (3.10) arc (-30:215:0.45);
        \end{tikzpicture}
        &
 \begin{tikzpicture}[scale = 0.85]
        \node[vertex] (1) at (0,0) {$1$};
        \node[vertex] (2) at (0,1.5) {$2$};
        \draw[->] (1.200) arc (-210:35:0.45);
        \draw[->,dashed] (1) to [bend left=30] (2);
        \draw[->,dashed] (2) to [bend left=30] (1);
        \draw[->] (2.10) arc (-30:215:0.45);
        \end{tikzpicture}
        \\
        \begin{tabular}{c}Fixed\\[-0.3em]point\end{tabular}
        & Mussel
        & \begin{tabular}{c}Parallel\\[-0.3em]loop\end{tabular}
        & \begin{tabular}{c}Anti-parallel\\[-0.3em]loop\end{tabular}
        & Necklace
        & Chain
        & Dumbbell
        \\
        $\mathcal{X}_2$
        & $\mathcal{M}_2$
        & $\mathcal{O}_2$
        & $\mathcal{A}_2$
        & $\mathcal{N}_2$
        & $\mathcal{C}_2$
        & $\mathcal{D}_2$
    \end{tabular}
    \caption{Types of structures in $\mathcal{H}_2$ (Hermitian case)}
    \label{tab:typesH20Herm}
\end{table}

\paragraph{Fixed points} Let $\mathcal{X}_2$ be the set of fixed points in $\mathcal{H}_2^0$. Then
\begin{equation}
    \mathcal{X}_2 = \textsc{Set} \bigg{(}\!\!\!
        \begin{tikzpicture}[baseline=3.7ex,scale = 0.85]
        \node[vertex] (1) at (-2.5,1) {$1$};
        \draw[->] (1.-60) arc (180+45+20:540-45-15:0.45);
        \draw[->,dashed] (1.-90) arc (180+45:540-45:0.6);
        \end{tikzpicture}
\bigg{)}.
\end{equation}
In terms of the EGF, this gives a factor of $\exp((c_1c_2+\kappa_2) t)$.

\paragraph{Mussels}
Next, let $\mathcal{M}_2$ be the set of mussels in $\mathcal{H}_2^0$. Then
\begin{equation}
    \mathcal{M}_2 = \textsc{Set}\Bigg{(}
        \begin{tikzpicture}[baseline=-0.65ex,scale = 0.85]
        \node[vertex] (1) at (-2.5,1) {$1$};
        \node[vertex] (2) at (-2.5,-1) {$2$};
        \draw[->] (1.-65) to [bend left=50] (2.65);
        \draw[->,dashed] (2) to [bend left=50] (1);
        \draw[->] (2.115) to [bend left=50] (1.-115);
        \draw[->,dashed] (1) to [bend left=50] (2);
        \end{tikzpicture}
\Bigg{)}.
\end{equation}
The EGF of the inner factor is $\lambda_{22} t^2/2!$ so the EGF of $\mathcal{M}_2$ is $\exp(\tfrac12 \lambda_{22} t^2)$.

\paragraph{Necklaces} For the set of necklaces $\mathcal{N}_2$,
\begin{equation}
    \mathcal{N}_2 = \textsc{Set} \Bigg{(} \frac12 \,  \textsc{Cyc}_{\geq 2} \Bigg{(}
    \obchauab\!\!\!\!\!\!
    + \!\!\! \obchauba \!\!\!\
\Bigg{)}\Bigg{)}.
\end{equation}
Hence, we have that its EGF is
\begin{equation}
\exp\left(\frac12\left(-2\lambda_{11}^2\frac{t^2}{2!}+\ln\frac{1}{1-2\lambda_{11}^2\frac{t^2}{2!}}\right)\right) = \frac{e^{-\frac12 \lambda_{11}^2 t^2}}{\sqrt{1-\lambda_{11}^2 t^2}}.   
\end{equation}

\paragraph{Chains} Let $\mathcal{C}_2$ be the set of all chains. We can write
\begin{equation}
    \mathcal{C}_2 =
        \textsc{Set} \Bigg{(} \ofixl \!\!     
\star  \textsc{Seq}_{\geq 1} \Bigg{(}
\obchadab\!\!\!\!\!\! + \!\!\!\obchadba\!\!\! \Bigg{)} \star \fixdar \Bigg{)}.
\end{equation}
In terms of the EGF,
\begin{equation*}
\exp\left[c_1\left(-1+\frac{1}{1-2\lambda_{11}^2 \frac{t^2}{2!}}\right) c_2 t \right]= \exp \left(\frac{c_1 c_2 \lambda_{11}^2t^3}{1-\lambda_{11}^2 t^2}\right).
\end{equation*}

\paragraph{Dumbbells} The set of dumbbels $\mathcal{D}_2$ can be constructed as
\begin{equation}
\begin{split}    
    \mathcal{D}_2 = \,\, &
        \textsc{Set} \Bigg{(}\,\frac12\,\,\,
        \oachaua \,\,     
        \star  \textsc{Seq} \Bigg{(}
        \obchauab \!\!\!\!\!\! +
        \!\!\!\obchauba\!\!\!\
        \Bigg{)}
        \star\,\, \fixuap \, \Bigg{)}
    \\
        \star \,\, & \textsc{Set} \Bigg{(}\,\frac12\,\,\, \oachada \,\,     
        \star  \textsc{Seq} \Bigg{(}
        \obchadab \!\!\!\!\!\!
        + \!\!\! \obchadba \!\!\!\
        \Bigg{)}
        \star\,\,
        \fixdap \, \Bigg{)}.
    \end{split}
    \end{equation}
Its EGF is thus equal to
\begin{equation}
\begin{split}
&\left[\exp\left(-\frac12 c_1 \lambda_{11}t\left(\frac{1}{1-2\lambda_{11}^2 \frac{t^2}{2!}}\right) c_1 t\right)\right]\left[\exp\left(-\frac12 c_2 \lambda_{11} t\left(\frac{1}{1-2\lambda_{11}^2 \frac{t^2}{2!}}\right) c_2 t\right)\right]\\
& = \exp\left(-\,\frac{\frac12 (c_1^2+c_2^2)\lambda_{11} t^2}{1-\lambda_{11}^2t^2}\right).
\end{split}
\end{equation}

\paragraph{Anti-parallel Loops} We can write
\begin{equation}
    \mathcal{A}_2 = \textsc{Set} \bigg{(} \textsc{Cyc}_{\geq 3} \bigg{(} \oalooduo \!\!\!\!\!\!
\bigg{)}\bigg{)}.
\end{equation}
As the element in the bracket has weight $W = \lambda_{11}t$, the corresponding EGF of $\mathcal{A}_2$ is
\begin{equation*}
\exp\left(-\lambda_{11}t - \frac{\lambda_{11}^2 t^2}{2} + \ln\frac{1}{1-\lambda_{11} t}\right) = \frac{e^{-\lambda_{11}t-\frac12 \lambda_{11}^2t^2}}{1-\lambda_{11}t}.
\end{equation*}

\paragraph{Parallel Loops}
For the set of parallel loops $\mathcal{O}_2$, observe that $\mathcal{O}_2 = \textsc{Set}(\mathcal{L})$ where each element $L \in \mathcal{L}$ is a single pair of parallel loops of length $n \geq 3$ and the weight of an element $L \in \mathcal{L}$ of order $n$ is 
\[
w_L = (-1)^{n-1}\lambda^{|\{e = (j,j') \in E(L): j' > j\}|}{\lambda'}^{|\{e = (j,j') \in E(L): j' < j\}|}
\]
where $\lambda = \lambda_{20}$ and $\lambda' = \lambda_{02}$. By Proposition \ref{prop:Bmotivation} and Proposition \ref{prop:Bexpression}, the EGF for $\mathcal{L}$ is 
\begin{equation}\label{eq:parallelloops}
B\left(\frac{\lambda_{20}}{\lambda_{02}},\lambda_{02}t\right) = \ln\left(\frac{\lambda_{02}-\lambda_{20}}{\lambda_{02}-\lambda_{20}e^{(\lambda_{02}-\lambda_{20})t}}\right) - \lambda_{20}t - \lambda_{20}\lambda_{02}\frac{t^2}{2}
\end{equation}
Exponentiating this gives the EGF for $\mathcal{O}_2$, which is 
\[
\frac{\lambda_{02}-\lambda_{20}}{\lambda_{02}-\lambda_{20}e^{(\lambda_{02}-\lambda_{20})t}}exp\left(-\lambda_{20}t - \lambda_{20}\lambda_{02}\frac{t^2}{2}\right).
\]
Note that when $\lambda_{20} = \lambda_{02}$ (as a limit), this is equal to
\begin{equation}
\frac{1}{1-\lambda_{20}t} \exp\left(-\lambda_{20}t  - \lambda_{20}^2\frac{t^2}{2}\right).\end{equation}
which is precisely the anti-parallel loop EGF with $\lambda_{20}$ instead of $\lambda_{11}$ as we would expect.

\paragraph{Conclusion}
As $\mathcal{H}_2^0 = \mathcal{X}_2 \star \mathcal{M}_2 \star \mathcal{O}_2 \star \mathcal{A}_2 \star \mathcal{N}_2 \star \mathcal{C}_2 \star \mathcal{D}_2$, by multiplying the EGFs we obtain that
\begin{equation*}
H_2^0(t) \!=\! \frac{\left(\lambda_{02}\!-\!\lambda _{20}\right) \exp \!\left(\!
   \left(c_1 c_2\!+\!\kappa_2\!-\!\lambda_{11}\!-\!\lambda_{20}\right)t\!+\!\frac{t^2}{2}\left(\lambda_{22}\!-\!2\lambda_{11}^2 \!-\! \lambda_{02}
   \lambda_{20} \right)\!+\!\frac{c_1 c_2 \lambda_{11}^2 t^3}{1\!-\!\lambda_{11}^2 t^2}\!-\!\frac{t^2}{2}\frac{\left(c_1^2\!+\!c_2^2\right) \lambda_{11}}{1\!-\!\lambda_{11}^2 t^2}\!\right)}{\left(1\!-\!
   \lambda_{11}t\right) \sqrt{1\!-\!\lambda_{11}^2t^2}
   \left(\lambda_{02}\!-\!\lambda_{20} e^{ \left(\lambda_{02}\!-\!\lambda_{20}\right)t}\right)}.
\end{equation*}

\subsubsection{One mark}
The one-marked structures are the same as in the case of Wigner matrices. There are 1-marked mussels, chains and dumbbells (see Table \ref{tab:typesH21} in the section on Wigner matrices).

\paragraph{Marked mussels}
Let $\mathcal{M}_2^1$ be the set of $1$-marked mussels. By adding up the contribution for the four different possibilities of marking, that is
\begin{equation}
\mathcal{M}_2^1 = \musmuab\, + \, \musmuba\, +\, \musmdab + \musmdba,
\end{equation}
we obtain $(a_1 + a_2)(\lambda_{21}+\lambda_{12})t^2/2!$ for its EGF.

\paragraph{Marked chains} For the set $\mathcal{C}_2^1$ of chains with one mark, we write
\begin{equation}
\begin{split}
        \mathcal{C}_2^1 = \, & \ofixmul\!\! \star \textsc{Seq} \Bigg{(} \obchadab\!\!\!\!\!\! + \!\!\!\obchadba\!\!\!\ \Bigg{)} \star \fixdar.
    \\
        + \, & \ofixmdl \!\! \star  \textsc{Seq} \Bigg{(} \obchauab\!\!\!\!\!\! + \!\!\!\obchauba\!\!\!\ \Bigg{)} \star \fixuar.
\end{split}
    \end{equation}
In terms of the EGF,
\begin{equation*}
a_1\left(\frac{1}{1-2\lambda_{11}^2 \frac{t^2}{2!}}\right) c_2 t + a_2\left(\frac{1}{1-2\lambda_{11}^2 \frac{t^2}{2!}}\right) c_1 t=  \frac{(a_1 c_2+a_2 c_1) t}{1-\lambda_{11}^2 t^2}.
\end{equation*}

\paragraph{Marked dumbbells} For the set $\mathcal{D}_2^1$ of all dumbbells with one mark,
\begin{equation}
\begin{split}
        \mathcal{D}_2^1 = \,\, & \oachamua \,\,\, \star \, \textsc{Seq} \Bigg{(} \obchauab\!\!\!\!\!\! + \!\!\!\obchauba\!\!\! \Bigg{)} \star\,\, \fixuap \\
        + \,\, & \oachamda\,\,\, \star \, \textsc{Seq} \Bigg{(} \obchadab\!\!\!\!\!\! + \!\!\!\obchadba\!\!\! \Bigg{)} \star\,\, \fixdap.
    \end{split}
    \end{equation}
Its EGF is thus equal to
\begin{equation}
    -a_1 \lambda_{11}t\left(\frac{1}{1-2\lambda_{11}^2 \frac{t^2}{2!}}\right) c_1 t -a_2 \lambda_{11}t\left(\frac{1}{1-2\lambda_{11}^2 \frac{t^2}{2!}}\right) c_2 t = -\,\frac{\lambda_{11}(a_1c_1+a_2c_2) t^2}{1-\lambda_{11}^2t^2}.
\end{equation}

\paragraph{Conclusion}
As $\mathcal{H}_2^1 = (\mathcal{M}^1_2 + \mathcal{C}^1_2 + \mathcal{D}^1_2) \star \mathcal{H}_2^0$, we obtain that 
\begin{equation}
H_2^1(t)= \left((a_1 + a_2) (\lambda_{21}+\lambda_{12}) \frac{t^2}{2} +\frac{(a_1c_2+a_2c_1) t}{1-\lambda_{11}^2 t^2} -\frac{\lambda_{11}(a_1c_1+a_2c_2) t^2}{1-\lambda_{11}^2t^2}\right) H_2^0(t).
\end{equation}

\subsubsection{Two marks}
The following Table \ref{tab:typesH22herm} shows representatives of all types of the remaining (connected) structures having exactly two marks. Note that this table is the same as Table \ref{tab:typesH22} in the Wigner case except for the fact we now distinguish between two types of loops.
\begin{table}[H]
    \centering
\begin{tabular}{ccc}
        \begin{tikzpicture}[scale = 0.85]
        \node[vertex] (1) at (-1,0.8) {$1$};
        \node[vertex] (2) at (1,0.8) {$2$}; 
        \node[vertex] (3) at (0,-1) {$3$};
        \draw[->] (2.182) to [bend left=-40] (1.-002);
        \draw[->,dashed] (1) to [bend left=35] node[fill=white,sloped]{$\!\!\!\times\!\!\!$} (2);
        \node[vertex] (x) at (0,1) {$\times$};
        \draw[->,dashed] (2) to [bend left=40] (3);
        \draw[->] (3.55) to [bend left=-40] (2.250);
        \draw[->,dashed] (3) to [bend left=40] (1);
        \draw[->] (1.290) to [bend left=-40] (3.130);
        \end{tikzpicture}    
    &        \begin{tikzpicture}[scale = 0.85]
        \node[vertex] (1) at (-1,0.8) {$1$};
        \node[vertex] (2) at (1,0.8) {$2$}; 
        \node[vertex] (3) at (0,-1) {$3$};
        \draw[->] (2.182) to [bend left=-40] (1.-002);
        \draw[->,dashed] (1) to [bend left=35] node[fill=white,sloped]{$\!\!\!\times\!\!\!$} (2);
        \node[vertex] (x) at (0,1) {$\times$};
        \draw[->,dashed] (2) to [bend left=40] (3);
        \draw[->] (3.55) to [bend left=-40] (2.250);
        \draw[->,dashed] (3) to [bend left=40] (1);
        \draw[->] (1.290) to [bend left=-40] (3.130);
        \end{tikzpicture}    
    &
        \begin{tikzpicture}[scale = 0.85]
        \node[vertex] (1) at (0,0) {$1$};
        \node[vertex] (2) at (1.5,1) {$2$};
        \node[vertex] (3) at (3,2) {$3$};
        \draw[->] (1.-70+180) arc (-130+180:125+180:0.45) node[fill=white,midway,sloped]{$\!\!\times\!\!$};
        \draw[->,dashed] (1) to [bend left=20] (2);
        \draw[->,dashed] (2) to [bend left=20] (1);
        \draw[->] (2) to [bend left=20] (3);
        \draw[->] (3) to [bend left=20] (2);
        \draw[->,dashed] (3.-70) arc (-130:125:0.45) node[fill=white,midway,sloped]{$\!\!\times\!\!$};
        \end{tikzpicture}
    \\
     $2$-Marked anti-parallel loop & $2$-Marked parallel loop & $2$-Marked chain
        \\
        $\mathcal{O}_2^2$
        & $\mathcal{A}_2^2$
        & $\mathcal{C}_2^2$
\end{tabular}
    \caption{Types of structures in $\mathcal{H}_2^2$ with two marks (Hermitian case)}
    \label{tab:typesH22herm}
\end{table}

\begin{remark}
There are no marked dumbbells as they would violate the assumption that there is at most one mark per one type of arrow. There are, however, $2$-marked fixed points (two marked $1$-cycle), but they are again counted as $2$-marked chains with one vertex. Similarly, $2$-marked mussels (two $2$-cycles with two marks) are counted as $2$-marked loops with two vertices.
\end{remark}

\paragraph{Marked anti-parallel loops} The set of marked anti-parallel loops $\mathcal{A}_2^2$ (including mussels) is created as follows
\begin{equation}
        \mathcal{A}^2_2 =
        \bigg{(}\oaloomdmuo\!\!\!\!\!\! \star \textsc{Seq}_{\geq 1} \bigg{(} \oalooduo
        \!\!\!\!\!\!
        \bigg{)}\bigg{)}.
    \end{equation}
Its EGF is then simply
\begin{equation*}
a_1a_2 t \left(-1+\frac{1}{1-\lambda_{11} t}\right)
= \frac{a_1 a_2 \lambda_{11} t^2}{1-\lambda_{11} t}.
\end{equation*}

\paragraph{Marked parallel loops} 
We can analyze the set of marked parallel loops as follows. For a given marked parallel loop of length $n$ with $k$ increasing edges and $n-k$ decreasing edges, the weight before we mark a pair of parallel edges is $\lambda_{20}^{k}\lambda_{02}^{n-k}$

When we mark this parallel loop, if we mark an increasing edge (i.e., an edge $j,j'$ such that $j' > j$) then this replaces a factor of $\lambda_{20}$ with ${a_1}{a_2}$. There are $k$ ways to do this so the total contribution is $k{a_1}{a_2}\lambda_{20}^{k-1}\lambda_{02}^{n-k} = {a_1}{a_2}\frac{\partial\left(\lambda_{20}^{k}\lambda_{02}^{n-k}\right)}{\partial{\lambda_{20}}}$. Similarly, if we mark a decreasing edge then this replaces a factor of $\lambda_{02}$ with ${a_1}{a_2}$. There are $n-k$ ways to do this so the total contribution is ${a_1}{a_2}(n-k)\lambda_{20}^{k-1}\lambda_{02}^{n-k-1} = {a_1}{a_2}\frac{\partial\left(\lambda_{20}^{k}\lambda_{02}^{n-k}\right)}{\partial{\lambda_{02}}}$.

Thus, we can obtain the contribution from marking a parallel loop by computing the partial derivatives of its original contribution with respect to $\lambda_{20}$ and $\lambda_{02}$, adding these partial derivates together, and multiplying the result by ${a_1}{a_2}$.

By equation \ref{eq:parallelloops}, the EGF for a single pair of parallel loops of length at least $3$ is
\[
B\left(\frac{\lambda_{20}}{\lambda_{02}},\lambda_{02}t\right) = \ln\left(\frac{\lambda_{02}-\lambda_{20}}{\lambda_{02}-\lambda_{20}e^{(\lambda_{02}-\lambda_{20})t}}\right) - \lambda_{20}t - \lambda_{20}\lambda_{02}\frac{t^2}{2}.
\]
We now observe that 
\begin{enumerate}
\item $\frac{\partial{B\left(\frac{\lambda_{20}}{\lambda_{02}},\lambda_{02}t\right)}}{\partial{\lambda_{20}}} = -\frac{1}{\lambda_{02}-\lambda_{20}} + \frac{(1 - \lambda_{20}t)e^{(\lambda_{02}-\lambda_{20})t}}{\lambda_{02}-\lambda_{20}e^{(\lambda_{02}-\lambda_{20})t}}-t-\lambda_{02}\frac{t^2}{2}$.
\item $\frac{\partial{B\left(\frac{\lambda_{20}}{\lambda_{02}},\lambda_{02}t\right)}}{\partial{\lambda_{02}}} = \frac{1}{\lambda_{02}-\lambda_{20}} + \frac{-1 + \lambda_{20}te^{(\lambda_{02}-\lambda_{20})t}}{\lambda_{02}-\lambda_{20}e^{(\lambda_{02}-\lambda_{20})t}}-\lambda_{20}\frac{t^2}{2}$.
\end{enumerate}
Adding these contributions together and multiplying the result by ${a_1}{a_2}$ gives \\
${a_1}{a_2}\left(\frac{e^{(\lambda_{02}-\lambda_{20})t} - 1}{\lambda_{02}-\lambda_{20}e^{(\lambda_{02}-\lambda_{20})t}} - t - (\lambda_{20} + \lambda_{02})\frac{t^2}{2}\right)$.

We make one final adjustment to this expression. Since we want to include marked parallel loops of length $2$ in $\mathcal{O}_2^2$, we add ${a_1}{a_2}(\lambda_{20} + \lambda_{02})\frac{t^2}{2}$ to the above expression. Thus, the EGF for $\mathcal{O}_2^2$ is 
\[
{a_1}{a_2}\left(\frac{e^{(\lambda_{02}-\lambda_{20})t} - 1}{\lambda_{02}-\lambda_{20}e^{(\lambda_{02}-\lambda_{20})t}} - t\right).
\]
\paragraph{Pair of one-marked structures} 
Finally, let $\mathcal{P}_2^2$ be the set of pairs of structures one with an undashed arrow marked and the other with a dashed arrow marked. Similarly as in the Wigner case,
\begin{equation}
\begin{split}
\mathcal{P}_2^2 = \,\,\, &
    \Bigg{[}\, \musmuab \,\, + \,\, \musmuba + \ofixmul\!\! \star \textsc{Seq} \Bigg{(} \obchadab\!\!\!\!\!\! + \!\!\!\obchadba\!\!\!\ \Bigg{)} \star \fixdar
    \\
    \,\, & + \oachamua \,\,\, \star \, \textsc{Seq} \Bigg{(} \obchauab\!\!\!\!\!\! + \!\!\!\obchauba\!\!\! \Bigg{)} \star\,\, \fixuap \Bigg{]}\\
    \star \, \, & \Bigg{[}\, \musmdab \,\, + \musmdba + \ofixmdl \!\! \star  \textsc{Seq} \Bigg{(} \obchauab\!\!\!\!\!\! + \!\!\!\obchauba\!\!\!\ \Bigg{)} \star \fixuar
    \\
    \,\, & + \oachamda\,\,\, \star \, \textsc{Seq} \Bigg{(} \obchadab\!\!\!\!\!\! + \!\!\!\obchadba\!\!\! \Bigg{)} \star\,\, \fixdap \Bigg{]}.
    \end{split}
\end{equation}
For its EGF, we thus have
\begin{equation*}
\begin{split}
&\left[a_1 (\lambda_{21}+\lambda_{12}) \frac{t^2}{2} +a_1\left(\frac{1}{1-2\lambda_{11}^2 \frac{t^2}{2!}}\right) c_2 t -a_1 \lambda_{11}t\left(\frac{1}{1-2\lambda_{11}^2 \frac{t^2}{2!}}\right) c_1 t \right]\\
& \left[a_2 (\lambda_{21}+\lambda_{12}) \frac{t^2}{2} + a_2\left(\frac{1}{1-2\lambda_{11}^2 \frac{t^2}{2!}}\right) c_1 t -a_2 \lambda_{11}t\left(\frac{1}{1-2\lambda_{11}^2 \frac{t^2}{2!}}\right) c_2 t\right],
\end{split}
\end{equation*}
which is, after simplification, equal to
\begin{equation*}
    a_1a_2\left((\lambda_{21}+\lambda_{12})\frac{t^2}{2}+\frac{c_2 t}{1-\lambda_{11}^2 t^2}-\frac{c_1 \lambda_{11} t^2}{1-\lambda_{11}^2 t^2}\right) \left((\lambda_{21}+\lambda_{12})\frac{t^2}{2}+\frac{c_1 t}{1-\lambda_{11}^2 t^2}-\frac{c_2 \lambda_{11}t^2}{1-\lambda_{11}^2 t^2}\right).
\end{equation*}

\paragraph{Conclusion}
Adding the contributions together, since $\mathcal{H}_2^2 = (\mathcal{A}_2^2+\mathcal{O}_2^2+\mathcal{C}_2^2+\mathcal{P}_2^2) \star \mathcal{H}_2^0$,
\begin{align*}
H_2^2(t) =  a_1 a_2 & \left[ \frac{\lambda _{11} t^2}{1-\lambda _{11} t}+\frac{t}{1-\lambda _{11}^2 t^2}-t+\frac{e^{\left(\lambda_{02}-\lambda_{20}\right) t}-1}{\lambda_{02}-\lambda_{20} e^{\left(\lambda_{02}-\lambda_{20}\right) t}}+\left(\left(\lambda_{21}+\lambda_{12}\right) \frac{t^2}{2} + \frac{c_2 t}{1\!-\!\lambda_{11}^2 t^2} \right.\right. \\
& \left.\left. \,\, - \,\frac{c_1 \lambda_{11} t^2}{1\!-\!\lambda_{11}^2 t^2}\right) \left(\left(\lambda_{21}+\lambda_{12}\right) \frac{t^2}{2}\!+\!\frac{c_1 t}{1\!-\!\lambda_{11}^2 t^2}\!-\!\frac{c_2 \lambda_{11} t^2}{1\!-\!\lambda_{11}^2 t^2}\right)\right]\!H_2^0(t).
\end{align*}

\subsubsection{Full generating function}
\begin{corollary}
Adding up the contributions, we get for any distribution $X_{ij}$ and $Z_{ii}$,
\begin{align*}
& H_2(t,a_1,a_2) = H_2^0(t)+H_2^1(t)+H_2^2(t) = \Bigg{[}1+(a_1 + a_2) (\lambda_{21}+\lambda_{12}) \frac{t^2}{2} +\frac{(a_1c_2+a_2c_1) t}{1-\lambda_{11}^2 t^2} \\
& - \left. \frac{\lambda_{11}(a_1c_1+a_2c_2) t^2}{1-\lambda_{11}^2t^2} + a_1 a_2 \Bigg{(}\frac{\lambda _{11} t^2}{1-\lambda _{11} t}+\frac{t}{1-\lambda _{11}^2 t^2}-t+\frac{e^{\left(\lambda_{02}-\lambda_{20}\right) t}-1}{\lambda_{02}-\lambda_{20} e^{\left(\lambda_{02}-\lambda_{20}\right) t}}+ \bigg{(}\left(\lambda_{21}+\lambda_{12}\right) \frac{t^2}{2} \right.\\
& + \frac{c_2 t}{1\!-\!\lambda_{11}^2 t^2} - \,\frac{c_1 \lambda_{11} t^2}{1\!-\!\lambda_{11}^2 t^2} \bigg{)} \bigg{(}\left(\lambda_{21}+\lambda_{12}\right) \frac{t^2}{2}\!+\!\frac{c_1 t}{1\!-\!\lambda_{11}^2 t^2}\!-\!\frac{c_2 \lambda_{11} t^2}{1\!-\!\lambda_{11}^2 t^2}\bigg{)}\Bigg{)}\Bigg{]} H_2^0(t).
\end{align*}
\end{corollary}

\printbibliography[heading=bibintoc]

\end{document}